\documentclass[a4paper,11pt]{article}
\usepackage{fullpage}
\usepackage{graphicx}
\usepackage[utf8]{inputenc}
\usepackage[english]{babel}
\usepackage{xcolor}

\usepackage{amsmath}
\usepackage{amssymb,amsthm,enumerate,url,tabularx}

\usepackage{psfrag}
\usepackage{graphicx}

\theoremstyle{plain}
\numberwithin{equation}{section}
\numberwithin{figure}{section}
\numberwithin{table}{section}

\newtheorem{theorem}{Theorem}[section]
\newtheorem{lemma}[theorem]{Lemma}
\newtheorem{corollary}[theorem]{Corollary}
\newtheorem{proposition}[theorem]{Proposition}

\newtheorem{assumption}[theorem]{Assumption}

\theoremstyle{definition}

\theoremstyle{remark}

\DeclareMathOperator{\supp}{supp}

\DeclareMathOperator{\dist}{dist}

\newcommand{\R}{\mathbb{R}}

\newcommand{\N}{\mathbb{N}}
\newcommand{\bk}{\mathbf{k}}

\newcommand{\bx}{\mathbf{x}}
\newcommand{\bn}{\mathbf{n}}

\newcommand{\bv}{\mathbf{v}}
\newcommand{\by}{\mathbf{y}}
\newcommand{\bp}{\mathbf{p}}
\newcommand{\bs}{\mathbf{s}}
\newcommand{\bxi}{\boldsymbol{\xi}}
\newcommand{\Z}{\mathbb{Z}}
\newcommand{\cA}{\mathcal{A}}
\newcommand{\cD}{\mathcal{D}}
\newcommand{\cP}{\mathcal{P}}
\newcommand{\cE}{\mathcal{E}}
\newcommand{\cF}{\mathcal{F}}
\newcommand{\cL}{\mathcal{L}}

\newcommand{\cS}{\mathcal{S}}
\newcommand{\cQ}{\mathcal{Q}}

\newcommand{\abs}[1]{\left|#1\right|}

\newcommand{\wilde}{\widetilde}
\newcommand{\what}{\widehat}

\newcommand{\cT}{\mathcal T}
\newcommand{\cV}{\mathcal V}
\newcommand{\Tref}{\widehat T}
\newcommand{\Pref}{\widehat P}
\newcommand{\Sref}{\widehat S}
\newcommand{\Kref}{\widehat K}
\newcommand{\tetref}{\widehat T^{3D}}

\title{On interpolation spaces of piecewise polynomials on mixed meshes\thanks{The author MK thankfully acknowledges supported by ANID Chile through project FONDECYT 1210579.
The authors JMM and AR thankfully acknowledge support by the Austrian Science Fund (FWF) through the SFB \emph{Taming complexity in
partial differential systems} (grant F65). AR thankfully acknowledges support by the stand-alone project P 36150, respectively.}}
\author{Michael Karkulik\thanks{Departamento de Matem\'atica, Universidad T\'ecnica Federico Santa Mar\'ia,
  Avenida Espa\~na 1680, Valpara\'iso, Chile, email: \texttt{michael.karkulik@usm.cl}},
Jens Markus Melenk\thanks{Institut f\"ur Analysis und Scientific Computing, Technische Universit\"at Wien,
    Wiedner Hauptstrasse 8-10, Wien, Austria, 
  email: \texttt{melenk@tuwien.ac.at}},
Alexander Rieder\thanks{Institut f\"ur Analysis und Scientific Computing, Technische Universit\"at Wien,
    Wiedner Hauptstrasse 8-10, Wien, Austria, 
  email: \texttt{alexander.rieder@tuwien.ac.at}},
}
\date{\today}

\begin{document}
\maketitle
\begin{abstract}  
  We consider fractional Sobolev spaces $H^\theta$, $\theta \in (0,1)$, on 2D domains
  and $H^1$-conforming discretizations by globally continuous piecewise polynomials on a mesh consisting of
  shape-regular triangles and quadrilaterals.
  We prove that the norm obtained from interpolating between the discrete space
  equipped with the $L^2$-norm on the one hand and the $H^1$-norm on the 
  other hand is equivalent to the corresponding continuous interpolation Sobolev norm, and the norm-equivalence
  constants are independent of meshsize and polynomial degree.
  This characterization of the Sobolev norm is then used to show an inverse
  inequality between $H^1$ and $H^{\theta}$.
\end{abstract}
\section{Introduction}
Fractional Sobolev spaces arise frequently in both analysis and numerical analysis of partial
differential or integral equations. As examples, we mention the classical trace space
$H^{1/2}(\partial\Omega)$ and its dual $H^{-1/2}(\partial\Omega)$ on the boundary
of some domain $\Omega$, which are basic function spaces
in the analysis of boundary intgral equations, or the more general spaces $H^\theta(\Omega)$ for
$\theta\in(0,1)$, which arise, e.g., in problems involving fractional diffusion processes. These
spaces can be characterized as interpolation spaces between $L^2$ and $H^1$, e.g.,
\begin{align*}
  H^\theta(\Omega) := [L^2(\Omega),H^1(\Omega)]_\theta := [L^2(\Omega), H^1(\Omega)]_{\theta,2},
\end{align*}
where we use the definition of interpolation spaces via the K-method, cf.~\cite{BerghL_76,Tartar,Triebel} and the details in
Section~\ref{sec:defs} below.
This characterization is especially convenient, as the so-called interpolation theorem allows
to extend mapping properties of linear operators $T:U^i\rightarrow V^i$ for $i=0,1$ to the case
$T:U^\theta \rightarrow V^\theta$, where $U^\theta:=[U^0,U^1]_\theta$ (same for $V$) is an interpolation space.
In the numerical analysis of the problems mentioned above,
in particular for Galerkin discretizations of partial differential or integral equations,
discrete (i.e. finite dimensional) subspaces $U_N\subset U^1$ are employed.
We use the notation $U_N^\theta = \left( U_N, \| \cdot \|_{U^\theta} \right)$, $\theta \in [0,1]$,
to emphasize that these spaces can be equipped with different norms.
The special case arises
where mapping properties of a linear operator $T$ can be derived exclusively on discrete spaces,
i.e., $T: U_N^i \rightarrow V_N^i$ for $i\in\left\{ 0,1 \right\}$.
Then, although one is ultimately interested in 
\begin{align*}
  T: U^\theta_N \rightarrow V^\theta_N,
\end{align*}
the interpolation theorem only states
\begin{align*}
  T: \left[ U_N^0,U_N^1 \right]_\theta
  \rightarrow
  \left[ V_N^0,V_N^1 \right]_\theta.
\end{align*}
By definition of interpolation spaces,
\begin{align*}
  \| v_N \|_{V_N^\theta} \leq \| v_N \|_{\left[ V_N^0,V_N^1 \right]_\theta} \qquad\text{ for all } v_N \in V_N.
\end{align*}
On the other hand, as $U_N$ is finite dimensional, the estimate
\begin{align}\label{eq:C}
  \| u_N \|_{\left[ U_N^0,U_N^1 \right]_\theta}
  \leq C \| u_N \|_{U_N^\theta}
  \qquad\text{ for all } u_N \in U_N
\end{align}
is certainly true for some constant $C=C_N>0$ depending on $N$.
A natural question in this situation is whether
$C_N$ is in fact independent of $N$, or, in other words, if the norm obtained by interpolating a discrete
space equipped with two norms is equivalent to the continuous interpolation norm, uniformly in the
discretization parameter.
From the aforegoing exposition of the problem it is clear that finite element inverse estimates
are an immediate application where such results are employed. We will prove certain inverse estimates in fractional order spaces
in Section~\ref{section:invest} below.
In Section~\ref{section:applications} we comment on various applications where results of this type are also employed.
\bigskip

One way to establish~\eqref{eq:C} with $C>0$ independent of $N$
is presented in~\cite{ArioliL_09}: Assuming that a projection $P_N:U^0\rightarrow U_N$ is available
that is simultaneously bounded in $U^0$ and $U^1$ uniformly in $N$,
then~\eqref{eq:C} is valid. In the context of $h$-version discretizations, common quasi-interpolation
operators can be used as $P_N$, e.g., the Scott-Zhang projector~\cite{ScottZ_90} in the case of
$U^0=L^2(\Omega)$, $U^1=H^1(\Omega)$, $U_N = \cS^1(\cT_h)$.
An application of this argument to $p$-version (or spectral) discretizations would
have to rely on the existence of projection operators that are simultaneously bounded in
$L^2(\Omega)$ and $H^1(\Omega)$, uniformly in the polynomial degree.
For the single-element case of tensor product elements,
such operators can indeed be constructed, cf. \cite{BernardiM_99,BraessPS_09}. An extension to elements not
having tensor product structure or to multi-element settings is not immediate. In the present work,
we do not construct such a projection operator, but rely on the characterization of interpolation spaces
as \textit{trace spaces}, cf.~\cite[Ch.~40]{Tartar}, stating that $[U^0,U^1]_\theta$ is the space
of traces at $0$ of suitable Banach-space valued functions
$U\in C\left( [0,\infty),U^1 \right)\cap C^1\left( (0,\infty),U^0 \right)$,
in particular
\begin{align*}
  \| U(0) \|_{[U^0,U^1]_s}^2 \sim \int_0^\infty z^{1-2\theta} \left( \| U(z) \|_{U^1}^2 +
  \| \partial_z U(z) \|_{U^0}^2 \right)\,dz.
\end{align*}
Using this characterization to show~\eqref{eq:C} requires to construct a linear operator
$\cL: U^0 \rightarrow C\left( [0,\infty],U^0 \right)$ with the following 3 properties:
\begin{itemize}
  \item[(i)] \textbf{lifting:} $\cL u_N (0)=u_N$,
  \item[(ii)] \textbf{boundedness:} $\int_0^\infty z^{1-2\theta} \left( \| \cL u(z) \|_{U^1}^2 +
    \| \partial_z \cL u(z) \|_{U^0}^2 \right)\,dz\lesssim \| u \|_{[U^0,U^1]_s}^2$,
  \item[(iii)] \textbf{conformity:} $\cL: U_N \rightarrow C([0,\infty],U_N)$.
\end{itemize}
In the single-element tensor-product case this avenue was successfully taken
in~\cite{Maday_89,BernardiDM_92,BenBelgacem_94,BernardiDM_10}.
We also refer to the exposition in~\cite{BernardiDM_07} and to works considering stable polynomial trace lifting operators in 
the $p$ and $hp$ setting, cf.~\cite{BabuskaCMP_91,MunozSola_97}.
\subsection{Contributions of the present work}
We will take on the multi-element case of
meshes $\cT$ consisting of triangles and/or quadrilaterals, which we assume only to be admissible
(i.e., no hanging nodes) and shape-regular.
We consider local polynomial degrees $\bp = \left( p_K \right)_{K\in\cT}$, and our discrete space
will be the space $U_N = \cS^{\bp,1}(\cT)$ of globally continuous, piecewise polyonomial functions
(possibly equipped with homogeneous Dirichlet boundary conditions, indicated by a tilde $\wilde\cS^{\bp,1}(\cT)$),
and continuous spaces $U^0 = L^2(\Omega)$, $U^1=H^1(\Omega)$
(possibly equipped with homogeneous Dirichlet boundary conditions $U_1=\wilde H^1(\Omega)$), respectively.
The main result of this work is then the following.
\begin{theorem}\label{thm:main}
  Let $\cT$ be a mesh of $\Omega$ that fulfills Assumption~\ref{assumption:refmaps}
  and $\bp$ be a degree distribution on $\cT$ which
  fulfills Assumption~\ref{assumption1}. Then, for $\theta\in(0,1)$ there holds
  \begin{align*}
    \left[ (\wilde \cS^{\bp,1}(\cT),\| \cdot \|_{L^2(\Omega)}), (\wilde \cS^{\bp,1}(\cT),\| \cdot \|_{\wilde H^1(\Omega)}) \right]_{\theta}
    = (\wilde \cS^{\bp,1}(\cT),\| \cdot \|_{[L^2(\Omega),\wilde H^1(\Omega)]_\theta}),\\
    \left[ (\cS^{\bp,1}(\cT),\| \cdot \|_{L^2(\Omega)}), (\cS^{\bp,1}(\cT),\| \cdot \|_{H^1(\Omega)}) \right]_{\theta}
    = (\cS^{\bp,1}(\cT),\| \cdot \|_{[L^2(\Omega),H^1(\Omega)]_\theta})
  \end{align*}
  with equivalent norms. The constants in the norm equivalences depend only on $\theta$ and the shape regularity constants of $\cT$.
\end{theorem}
As mentioned in the beginning, interpolation spaces are defined via the K-method. Details are given in Section~\ref{sec:defs} below.
\subsection{Construction of the lifting operator}\label{sec:construction}
The proof of theorem~\ref{thm:main} will be given in Section~\ref{sec:lifting:dcomp} below.
It relies on the characterization of
interpolation spaces as trace spaces as given above, and hence on the definition of an appropriate lifting
operator $\cL$ with the properties given above.
We will construct the action of this lifting operator in several steps.
In a first step we will construct a single-element lifting operator $\cA$,
mapping functions from the reference triangle $\what T$ to the reference tetrahedron $\tetref$,
cf. Lemma~\ref{lemma:lifting-from-triangle}.
This will be done by an ubiquitous averaging process, cf.~\cite{Grisvard}, which goes back
at least to~\cite{Gagliardo_57}. This averaging process ensures the lifting property (i) and the boundedness property
(ii) \textit{locally}.
In a second step, homogeneous
Dirichlet boundary conditions on one or more edges $\what\cE$ of $\partial\what T$ will be taken into account,
leading to liftings $\cA_{\what\cE}$
that vanish on the associated faces of $\tetref$, cf. Lemma~\ref{lemma:lifting-from-triangle:bc}.
Finally, a Duffy transform will
be used to transform the refence tetrahedron $\tetref$ into a prism $\what P =\what T\times (0,1)$, which gives
rise to the associated lifting operator $\cA_{\what\cE}^{\what P}$, cf. Lemma~\ref{lemma:lifting-to-prism}.
Applying this operator 
elementwise, we can construct conforming liftings of “simple” discrete functions $u_{hp}\in\cS^{\bp,1}(\cT)$,
i.e., functions with local support (elements, edge- or vertex-patches) which on every element of their support
are copies of a certain “symmetric” reference function. Consequently, we can lift such simple discrete functions
in a conforming way to ensure property (iii).
It is obvious that $\cS^{\bp,1}(\cT)$ can be decomposed on an algebraic level
into such “simple” discrete functions by successively subtracting the degrees of
freedom associated with vertices and edges. However, in order to combine the bounded local lifting operators
into a globally bounded one, we need a decomposition which additionally is bounded in the norms of interest.
To that end, we will employ a result of our recent work~\cite{KMR_19}, cf. Lemma~\ref{lem:kmr:decomp:2} below.
\subsection{Applications}\label{section:applications}
As already stated above, the result of Theorem~\ref{thm:main} can be used to
prove finite element inverse estimates
on fractional order spaces, cf. Section~\ref{section:invest}.
Such inverse estimates are widely employed in finite and boundary element analysis,
cf.~\cite{GrahamHS_05,Georgoulis_08}.
Various other available results in the literature rely on our main result Theorem~\ref{thm:main}.
We give a brief overview:
\begin{itemize}
  \item[(i)] In~\cite{AuradaFFKMP_17}, inverse estimates for the classical boundary integral operators
    associated to the three-dimensional Laplacian are derived in the $hp$-setting based on the presently proved Theorem~\ref{thm:main},
    cf.~\cite[Cor.~3.2]{AuradaFFKMP_17}.
  \item[(ii)] In~\cite{FuhrerMPR_15}, $p$-explict bounds on the condition number of $hp$-boundary 
    element methods are derived, with Theorem~\ref{thm:main} proving crucial to obtain the 
    needed estimates in fractional Sobolev norms.
  \item[(iii)] When considering discretizations of parabolic problems,
     (analytic) semigroups are the natural setting. In this general theory,
     the regularity of the solution is governed
     by the regularity of the initial condition, as represented by different interpolation spaces.

     When considering discretizations of such problems via the method of lines,
     i.e., by first performing a discretization of the spatial variables, one ends
     up with a semidiscrete semigroup~\cite[Chapter 9]{Thom97}, and the 
     corresponding interpolation spaces are between spaces of piecewise polynomial function.
     Thus, when discretizing in time, studying these spaces becomes crucial to determine speed of
     convergence; see, e.g., \cite{SchSch01}.

     This approach is used in~\cite{MelenkR_21}, where an $hp$-finite element method with DG-Galerkin timestepping is 
     derived for a parabolic equation with fractional Laplace operator in space, and exponential convergence
     in the number of degrees of freedom is shown.
     Theorem~\ref{thm:main} provides the crucial regularity for the spatially discrete semigroup.
\end{itemize}

\section{Notation and preliminairies}
The shorthand $a \lesssim b$ expresses $a \leq C b$ for a constant $C > 0$ that does not depend on parameters of 
interest (in particular the mesh size $h$ and the polynomial degree $p$). The notation  
$a \sim b$ is short for $a \lesssim b$ in conjunction with $b \lesssim a$.
\subsection{Functional setting}\label{sec:defs}
Let $\omega\subset\R^2$ be polygonal domain, i.e., in particular Lipschitz and bounded.
The Sobolev spaces $L^2(\omega)$ and $H^1(\omega)$ are defined in a standard way, cf.~\cite{Adams_75,Tartar}.
For $\gamma\subset\partial\omega$ a relatively open subset of the boundary $\partial\omega$,
we also define $\wilde H^1_\gamma(\Omega)$ as functions in $H^1(\Omega)$ with vanishing trace on $\gamma$,
and we use the standard notation $\wilde H^1(\omega)=\wilde H^1_{\partial\omega}(\omega)$.
Fractional Sobolev spaces for $\theta\in(0,1)$ are defined in two ways. The first way is based on the
K-method of interpolation, cf.~\cite{BerghL_76,Tartar, Triebel}.
If $(X_0,\|\cdot\|_0)$ and $(X_1,\|\cdot\|_1)$ are two Banach spaces with continuous
embedding $X_1\subset X_0$, define the K-functional
$K^2(t,u) := \inf_{v\in X_1}\| u -v \|_0^2 + t^2\| v \|_1^2$ and the norm
\begin{align*}
  \| u \|_{[X_0,X_1]_\theta}^2 := \int_0^\infty t^{-2\theta} K^2(t,u)\frac{dt}{t}.
\end{align*}
Then, we define the interpolation space
\begin{align*}
  [X_0,X_1]_\theta := \left\{ u\in X_0 \mid \| u \|_{[X_0,X_1]_\theta}<\infty \right\}.
\end{align*}
We will use the spaces $[L^2(\omega),H^1(\omega) ]_\theta$
and $[ L^2(\omega),\wilde H^1_\gamma(\omega) ]_\theta$.
We will also use seminorms defined by interpolation: if $\| \cdot \|_1^2 = \| \cdot \|_0^2 + |\cdot|_1^2$ with $|\cdot|_1$ being a seminorm,
then let $k^2(t,u) := \inf_{v\in X_1}\| u -v \|_0^2 + t^2 | v |_1^2$, and define
\begin{align*}
  | u |_{[X_0,X_1]_\theta}^2 := \int_0^\infty t^{-2\theta} k^2(t,u)\frac{dt}{t}.
\end{align*}

The second way to define fractional order Sobolev spaces is by using 
Aronstein-Slobodeckij double integral norms.
Define
\begin{align*}
  \| u \|_{H^\theta(\omega)}^2 &:= \| u \|_{L^2(\omega)}^2 + | u |_{H^\theta(\omega)}^2,
  \qquad
  | u |_{H^\theta(\omega)}^2 := \int_{\omega}\int_{\omega} \frac{|u(x)-u(y)|^2}{|x-y|^{2+2\theta}}\,dxdy,
\end{align*}
and set
\begin{align*}
  H^\theta(\omega) &:= \left\{ u\in L^2(\omega)\mid \| u \|_{H^\theta(\omega)}<\infty \right\}.
\end{align*}
Norms defined by interpolation and by double integrals are equivalent on fixed Lipschitz domains,
cf.~\cite{McLean_00}:
\begin{lemma}\label{lem:refel:norms}
  Let $\what\omega\subset\R^2$ be a bounded Lipschitz domain and $\theta\in(0,1)$. Then,
  \begin{align*}
    H^\theta(\what\omega) &= [L^2(\what\omega),H^1(\what\omega) ]_\theta,
  \end{align*}
  with equivalent norms.
\end{lemma}
\subsection{Discrete setting} 
For finite sets $M$, we denote by $\#M$ the counting measure, i.e., the number of elements in $M$.
For geometric objects $M$, we denote by $h_M$ the Euclidean diameter of $M$, by $d_M$ the Euclidean
distance to $M$, and by $|M|$ the Lebesque measure of $M$.
We consider finite partitions (\textit{meshes})
$\cT$ of $\Omega$ into triangles and/or quadrilaterals $K$ (\textit{elements}),
which we define to be open sets.
We will use the reference triangle $\Tref$ given by the vertices
$\widehat\bv_1 = (0,\frac{2}{\sqrt{3}})$, $\widehat\bv_2 = (1,-\frac{1}{\sqrt{3}})$, $\widehat\bv_3=(-1,-\frac{1}{\sqrt{3}})$.
The reference rectangle is given by
\begin{align*}
  \Sref := \left\{ (\xi,\eta)\mid -1<\xi<1, -1/\sqrt{3} < \eta < 2/\sqrt{3} \right\}.
\end{align*}
The set of all vertices of $\cT$ is denoted by $\cV$ and the set of all edges is denoted by $\cE.$
Additionally, we will use the set $\cV^{\rm int}$ of inner vertices,
i.e., vertices not lying on the
boundary of $\Omega$ (again the same definition for inner edges).
We assume the following.
\begin{assumption}\label{assumption:refmaps}
  \begin{enumerate}
    \item For each element $K\in\cT$, there exists $\what K\in\{ \Sref,\Tref \}$ and a bijective element map
      $F_K:\overline{\what K}\rightarrow \overline K$ that is $C^1(\overline{\what K})$.
    \item The Gramian $G(x) := (F_K'(x))^\top F_K'(x)$ has two eigenvalues  which fulfill
      \begin{align*}
	\sup_{x\in\what K} \max\left( 
	  \frac{h_K^2}{\lambda_1(x)}, \frac{\lambda_1(x)}{h_K^2},
	  \frac{h_K^2}{\lambda_2(x)}, \frac{\lambda_2(x)}{h_K^2}
	\right) \leq\gamma
      \end{align*}
      for some fixed constant $\gamma>0$.
    \item The intersection $\overline{K_1}\cap\overline{K_2}$ of two distinct elements is either empty, exactly one point, 
      or exactly one edge. If the intersection is an edge $e = F_{K_1}(\what e_j) = F_{K_2}(\what e_k)$, then
      $F_{K_1}^{-1}\circ F_{K_2}|_{\what e_k}: \what e_k \rightarrow \what e_j$ is an affine bijection.
    \item Each edge is either fully contained in $\Omega$ or in $\partial\Omega$.
  \end{enumerate}
\end{assumption}
For a vertex $V\in\cV$ we define the vertex patch $\omega_V$ to be the (open) domain covered by all elements having $V$ as vertex,
\begin{align*}
  \omega_V = \textrm{interior}\Bigl( \bigcup_{V\in\overline K} \overline K \Bigr),
\end{align*}
and likewise we define edge patches.
We stress that we will frequently abuse this notation and use patches as collection of the
elements defining them.
For $p\in\N$ we define polynomial spaces
\begin{align*}
  \cP^p &:= \textrm{span}\left\{ x^i y ^j \mid 0 \leq i,j \text{ and }  i + j\leq p \right\},\\
  \cQ^p &:= \textrm{span}\left\{ x^i y ^j \mid 0 \leq i,j \leq p \right\}.
\end{align*}
We also write $\wilde \cP^p$ and $\wilde \cQ^p$ to indicate polynomials vanishing on the boundary of $\Sref$.
We write
\begin{align*}
  \cP^p(\Kref) :=
  \begin{cases}
    \cP^p & \text{ if } \Kref = \Tref,\\
    \cQ^p & \text{ if } \Kref = \Sref.
  \end{cases}
\end{align*}
For each element $K\in\cT$ we choose a polynomial degree $p_K\in\N$ and collect them in the family $\bp := \left( p_K \right)_{K\in\cT}$.
We define spaces of piecewise polynomials as
\begin{align*}
  \cS^{\bp,0}(\cT) &:= \left\{ u\in L^2(\Omega) \mid u\circ F_K \in \cP^{p_K}(\Kref) \text{ for all } K\in\cT  \right\},\\
  \cS^{\bp,1}(\cT) &:= \cS^{\bp,0}(\cT)\cap H^1(\Omega),\\
  \wilde\cS^{\bp,1}(\cT) &:= \cS^{\bp,0}(\cT)\cap \wilde H^1(\Omega).
\end{align*}
For subpartitions $\cT|_\omega\subset\cT$ of elements belonging to the patch $\omega$, we define
$\cS^{\bp,0}(\cT|_\omega)$,
$\cS^{\bp,1}(\cT|_{\omega})$, and $\wilde\cS^{\bp,1}(\cT|_{\omega})$ accordingly:
\begin{align*}
  \cS^{\bp,0}(\cT|_\omega) &:= \left\{ u\in L^2(\omega) \mid u\circ F_K \in \cP^{p_K}(\Kref) \text{ for all } K\in\omega  \right\},\\
  \cS^{\bp,1}(\cT|_\omega) &:= \cS^{\bp,0}(\cT|_\omega)\cap H^1(\omega),\\
  \wilde\cS^{\bp,1}(\cT|_\omega) &:= \cS^{\bp,0}(\cT_\omega)\cap \wilde H^1(\omega).
\end{align*}
In addition, we will assume that the polynomial degree distributions of our meshes satisfy the following assumption.
\begin{assumption}\label{assumption1}
  If a triangle $T$ and a quadrilateral $S$ share an edge $e$,
  then the corresponding polynomial degrees $p_T$ and $p_S$ satisfy
  \begin{align*}
    p_T\leq p_S \quad\text{ or }\quad 2p_S \leq p_T.
  \end{align*}
\end{assumption}
We will switch between $\Tref$ and $\Sref$ using the Duffy-transformation
\begin{align*}
  T_\cD:
  \begin{cases}
    \Sref\rightarrow\Tref,\\
    (\xi,\eta)\mapsto 
    \left(
    \frac{\frac{2}{\sqrt{3}}-\eta}{\sqrt{3}}\xi,\eta
    \right)
  \end{cases}
\end{align*}
which collapses the upper
edge of $\Sref$ into the vertex $\widehat \bv_1$. We will make use of the notation
$\Tref_\varepsilon := \{ (\xi,\eta)\in\Tref \mid d_{\widehat\bv_1}(\xi,\eta)< \varepsilon \}$.
Whenever $\Tref_\varepsilon$ shows up, it is assumed implicitly that $\varepsilon$ is small enough to ensure that
$\Tref_\varepsilon$ has positive distance to $\widehat\bv_2,\widehat\bv_3$.
The proof of the following lemma follows from elementary considerations, cf.~\cite[Lemma 5.5]{KMR_19}.
\begin{lemma}\label{lem:2dduffy}
  With the Duffy-transformation $T_\cD$ define the Duffy operator $\cD: u \mapsto u\circ T_\cD$.
  Then we have the estimates
      \begin{align*}
	\| \cD u \|_{L^2(\Sref)} &\lesssim \| d_{\widehat \bv_1}^{-1/2} u \|_{L^2(\Tref)}
	\lesssim \| u \|_{L^2(\Tref)} + \| u \|_{L^\infty(\Tref_\varepsilon)},\\
	\| \nabla\cD u \|_{L^2(\Sref)} &\lesssim \| d_{\widehat \bv_1}^{-1/2} \nabla u \|_{L^2(\Tref)}
	\lesssim \| \nabla u \|_{L^2(\Tref)} + \| \nabla u \|_{L^\infty(\Tref_\varepsilon)}.
      \end{align*}
\end{lemma}
We naturally embed $\R^2$ into $\R^3$ and identify $\Tref$, $\Sref$, $\widehat\bv_1$, $\widehat\bv_2$, $\widehat\bv_3$
as objects in $\R^3$.
We fix a reference tetrahedron $\tetref$ with top vertex $\widehat\bv_4=(0,0,1)$ and
bottom face $\Tref$, cf. Figure~\ref{fig:tetref}.
The lateral edge connecting $\widehat\bv_j$ and $\widehat\bv_4$ is called $\widehat e_j$, and the face opposite to $\widehat\bv_j$
is called $\widehat f_j$.
The edge that $\widehat f_j$ shares with $\Tref$ is called $\widehat e_{3+j}$.
For $\varepsilon>0$, we denote $\tetref_\varepsilon := \{ (\bx,z)\in\tetref \mid z > \varepsilon \}$.
\begin{figure}[htb]
  \centering
  \psfrag{v1}{$\widehat\bv_1$}
  \psfrag{v2}{$\widehat\bv_2$}
  \psfrag{v3}{$\widehat\bv_3$}
  \psfrag{v4}{$\widehat\bv_4$}
  \psfrag{e1}{$\widehat e_1$}
  \psfrag{e2}{$\widehat e_2$}
  \psfrag{e3}{$\widehat e_3$}
  \psfrag{e4}{$\widehat e_4$}
  \psfrag{e5}{$\widehat e_5$}
  \psfrag{e6}{$\widehat e_6$}
  \psfrag{f1}{$\widehat f_1$}
  \psfrag{f2}{$\widehat f_2$}
  \psfrag{f3}{$\widehat f_3$}
  \psfrag{x}{$x$}
  \psfrag{y}{$y$}
  \psfrag{T}{$\Tref$}
  \includegraphics[width=0.8\textwidth]{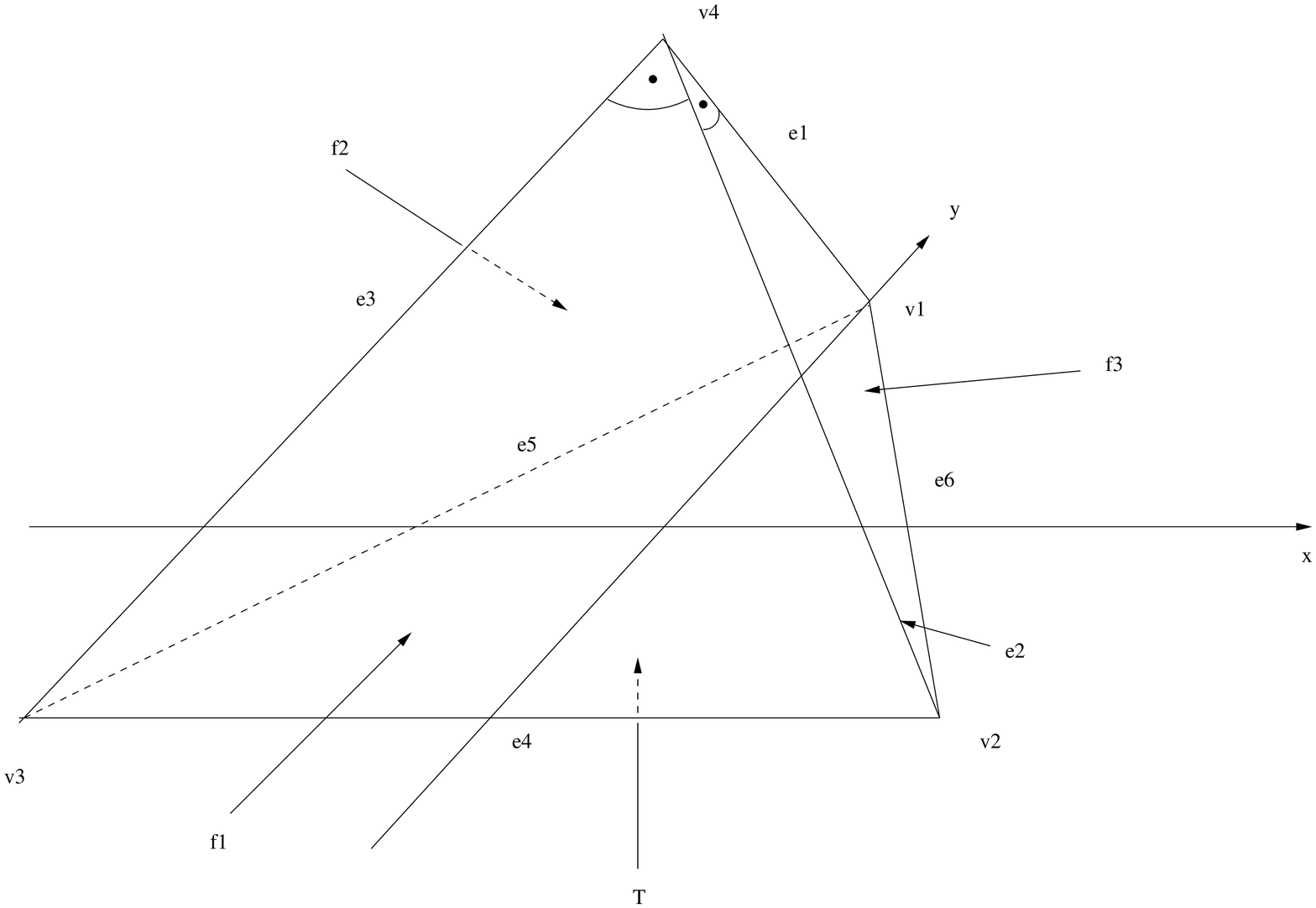}
  \caption{Reference tetrahedron $\tetref$. The top vertex $\widehat \bv_4$ is a right-angled corner.}
  \label{fig:tetref}
\end{figure}

Let $\Pi_{\widehat e_k}, \Pi_{\widehat f_k}:\R^3\rightarrow\R^3$ be the affine functions calculating the orthogonal projections onto
the line spanned by $\widehat e_k$ and the plane spanned by $\widehat f_k$, respectively.
Note that the lateral edges $\widehat e_j$, $j\in\left\{ 1,2,3 \right\}$, intersect in $\widehat\bv_4$ at right angles,
  which implies
  \begin{align}\label{eq:ort}
    \Pi_{\widehat e_j}|_{f_j} = \widehat\bv_4 \qquad\text{ for } j\in\left\{ 1,2,3 \right\}.
  \end{align}
  Furthermore, for $k\neq j$ and $\widehat e_{\ell_{j,k}}$ being the lateral edge shared by $\widehat f_j$ and $\widehat f_k$,
  \begin{align}\label{eq:ort2}
    \Pi_{\what f_k}|_{\what f_j} = \Pi_{\what e_{\ell_{j,k}}}|_{\what f_j}.
  \end{align}
The next result collects statements which are needed later on.
\begin{lemma}\label{lem:ort}
  For $j,k\in\left\{ 1,2,3 \right\}$ and $j\neq k$ there holds for $\alpha,\beta\in\R$ with $-1/2<\beta$,
  \begin{align}
    \| d_{\widehat\bv_j}^{\alpha} d_{\Tref}^\beta v \circ \Pi_{\widehat e_j} \|_{L^2(\tetref)}
    +
    \| d_{\widehat\bv_j}^{\alpha}d_{\Tref}^{1/2+\beta} v \circ \Pi_{\widehat e_j} \|_{L^2(\widehat f_k)}
    &\sim \| d_{\widehat \bv_j}^{1+\alpha+\beta} v \|_{L^2(\widehat e_j)}\label{eq:lemma:lifting-from-triangle:bc:1},\\
    \| d_{\widehat e_{3+k}}^\alpha d_{\Tref}^\beta v\circ \Pi_{\widehat f_k} \|_{L^2(\tetref)}
    &\sim \| d_{\widehat e_{3+k}}^{1/2+\alpha+\beta} v \|_{L^2(\widehat f_k)}\label{eq:lemma:lifting-from-triangle:bc:2},
  \end{align}
  where we recall that $\widehat e_{3+k}$ is the edge that $\widehat f_k$ shares with $\Tref$.
  Furthermore, we have for $k\in\left\{ 1,2,3 \right\}$ and $-1/2<\beta$ the equivalence
  \begin{align}
    \| d_{\Tref}^{\beta} v \circ \Pi_{\widehat e_k} \|_{L^2(\widehat f_k)} \sim |v(\widehat \bv_4)|\label{eq:lemma:lifting-from-triangle:bc:16}.
  \end{align}
\end{lemma}
\begin{proof}
  We may as well rotate and scale the setting to be able to use the standard Cartesian coordinate system
  and the orthogonal projection $\Pi_x$ onto the x-axis to show
  \begin{align*}
    \| d_{\widehat\bv_j}^{\alpha}d_{\Tref}^\beta v \circ \Pi_{\widehat e_j} \|_{L^2(\tetref)}^2
    &\sim \int_{x=0}^1 \int_{y=0}^{1-x}\int_{z=0}^{1-x-y} (1-x)^{2\alpha}(1-x-y-z)^{2\beta} \left( v\circ \Pi_x (x,y,z) \right)^2 \,dzdydx\\
    &= \int_{x=0}^1 (1-x)^{2\alpha}v(x,0,0)^2 \int_{y=0}^{1-x}\int_{z=0}^{1-x-y} (1-x-y-z)^{2\beta} \,dzdydx\\
    &= \frac{1}{(2\beta+1)(2\beta+2)} \int_{x=0}^1 v(x,0,0)^2 (1-x)^{2\alpha + 2\beta+2}\,dx \sim \| d_{\widehat \bv_j}^{1+\alpha+\beta} v \|_{L^2(\widehat e_j)}^2.
  \end{align*}
  The remaining results
  in~\eqref{eq:lemma:lifting-from-triangle:bc:1}--\eqref{eq:lemma:lifting-from-triangle:bc:2}
  follow the same way.
  The result~\eqref{eq:lemma:lifting-from-triangle:bc:16} follows from the observation~\eqref{eq:ort}.
\end{proof}
\section{Applications}
\subsection{Finite element inverse estimates}\label{section:invest}
Inverse estimates are a common tool in the analysis of numerical methods. They allow for
control of Sobolev norms of discrete functions via weighted versions of weaker norms.
If the norms involved are of integer order (including dual ones), such estimates are usually
proven locally using norm equivalence on finite dimensional spaces.
Consequently, arguments are more involved in the case of fractional order norms, as those are non-local.
A widely used approach here is to characterize fractional order Sobolev spaces as interpolation spaces,
cf. Section~\ref{sec:defs}. In this case it is necessary to know that the interpolation norm obtained
from the discrete spaces is equivalent to the norm obtained by interpolating the continuous spaces.
In the case of $hp$-finite element spaces, our main Theorem~\ref{thm:main} applies.\\

For a positive, measurable function $w$ on $\Omega$ we define $L^2(\Omega;\omega)$ to be the space of functions with finite norm
$\| f \|_{L^2(\Omega;\omega)}^2 := \int_\Omega f(x)^2\omega(x)\;dx.$
The next result can be found in~\cite[Lemma~23.1]{Tartar}.
\begin{proposition}\label{prop:L2}
  Let $\omega_0,\omega_1$ be positive, measurable functions on $\Omega$. For $\theta\in(0,1)$ it holds that
  \begin{align*}
    [L^2(\Omega;\omega_0),L^2(\Omega;\omega_1)]_{\theta} = L^2(\Omega;\omega_0^{1-\theta}\omega_1^\theta)
  \end{align*}
  with equivalent norms, with constants depending only on $\theta$.
\end{proposition}
For a mesh $\cT$ of $\Omega$ and a space $\cS^{\bp,1}(\cT)$, we define functions $h,p\in L^\infty(\Omega)$ by $h|_K = h_K$ and $p|_K = \bp_K$.
\begin{lemma}\label{lem:invest}
  For $\theta\in[0,1]$ there exists a constant $C>0$, such that for any space $\cS^{\bp,1}(\cT)$ it holds
  \begin{align*}
    \| h^{1-\theta} p^{-2(1-\theta)} \nabla u_h \|_{L^2(\Omega)} \leq C \| u_h \|_{[L^2(\Omega),H^1(\Omega)]_\theta}\quad
    \text{ for all } u_h\in \cS^{\bp,1}(\cT).
  \end{align*}
  Furthermore, this result is also valid if we use the discrete space $\wilde\cS^{\bp,1}(\cT)$ and the norm
  \\ $\| \cdot \|_{[L^2(\Omega),\wilde H^1(\Omega)]_\theta}$.
\end{lemma}
\begin{proof}
  Obviously, $\| \nabla w_h \|_{L^2(\Omega)}\leq \| w_h \|_{H^1(\Omega)}$ for all $w_h \in\cS^{\bp,1}(\cT)$. On the other hand,
  \begin{align}\label{lem:invest:eq1}
    \| h p^{-2} \nabla w_h \|_{L^2(\Omega)} \leq C \| w_h \|_{L^2(\Omega)},
  \end{align}
  which follows from simple scaling arguments, combined with a p-explicit inverse estimate on
  the reference element, cf.~\cite[Thm.~4.76]{Schwab_98}. Hence, 
  \begin{align*}
    \| h^{1-\theta} p^{-2(1-\theta)} \nabla u_h \|_{L^2(\Omega)}^2 &= 
    \sum_{j=1}^2 \| h^{1-\theta} p^{-2(1-\theta)} \partial_j u_h \|_{L^2(\Omega)}^2\\
    &\hspace{-.48cm}\stackrel{\text{Prop.}~\ref{prop:L2}}{\lesssim} \sum_{j=1}^2
    \int_0^\infty t^{-2\theta-1} \inf_{v\in L^2(\Omega)} \| hp^{-2} (\partial_j u_h-v) \|_{L^2(\Omega)}^2 + t^2\| v \|_{L^2(\Omega)}^2\,dt\\
    &\leq\sum_{j=1}^2
    \int_0^\infty t^{-2\theta-1} \inf_{v_h\in\cS^{\bp,1}(\cT)} \| hp^{-2}\partial_j (u_h-v_h) \|_{L^2(\Omega)}^2 + t^2\| \partial_j v_h \|_{L^2(\Omega)}^2\,dt\\
    &\hspace{-.15cm}\stackrel{\eqref{lem:invest:eq1}}{\leq}
    \int_0^\infty t^{-2\theta-1} \inf_{v_h\in\cS^{\bp,1}(\cT)} \| u_h-v_h \|_{L^2(\Omega)}^2 + t^2\| v_h \|_{H^1(\Omega)}^2\,dt.
  \end{align*}
  On the right-hand side, we have the norm of $u_h$ in the space
  $\left[ (\cS^{\bp,1}(\cT),\| \cdot \|_{L^2(\Omega)}), (\cS^{\bp,1}(\cT),\| \cdot \|_{H^1(\Omega)}) \right]_{\theta}$, which
  is equivalent to $\| u_h \|_{[L^2(\Omega),H^1(\Omega)]_\theta}$ according to Theorem~\ref{thm:main}.
  The last statement of the lemma clearly follows from the first one as
  $[L^2(\Omega),\wilde H^1(\Omega)]_\theta\subset[L^2(\Omega),H^1(\Omega)]_\theta$.
\end{proof}
In the following, we assume a global quasi-uniform mesh and constant polynomial degree distribution, that is
\begin{align*}
  \max_{K\in\cT} h_K \leq C_{\rm qu} \min_{K\in\cT} h_K 
\end{align*}
for some constant $C_{\rm qu}>0$, and $p_K = p$ for all $K\in\cT$.
\begin{corollary}
  Let $\cT$ be a quasi-uniform mesh and $\wilde\cS^{\bp,1}(\cT)$ have constant polynomial degree distribution.
  Let $0 \leq \theta\leq\mu\leq1$. Then there exists a constant $C>0$ (depending on $C_{\rm qu}$) such that
  \begin{align*}
    \frac{h^{\mu-\theta}}{p^{2(\mu-\theta)}}
    \| u_h \|_{[L^2(\Omega),H^1(\Omega)]_\mu}
    \leq C \| u_h \|_{[L^2(\Omega),H^1(\Omega)]_\theta}
     \quad\text{ for all } u_h\in \cS^{\bp,1}(\cT).
  \end{align*}
\end{corollary}
\begin{proof}
  According to Lemma~\ref{lem:invest} and obvious bounds,
  $h^{1-\theta}p^{-2(1-\theta)}\| u_h \|_{H^1(\Omega)} \leq C \| u_h \|_{[L^2(\Omega),H^1(\Omega)]_\theta}$.
  Note that the reiteration theorem~\cite[Thm.~3.5.3]{BerghL_76} gives $[L^2(\Omega),H^1(\Omega)]_\mu = [ H^\theta(\Omega),H^1(\Omega) ]_{s}$
  with $s=(\mu-\theta)/(1-\theta)$. Hence, {by common interpolation estimates, cf.~\cite[Thm.~1.3.3]{Triebel},}
  \begin{align*}
    \| u_h \|_{[L^2(\Omega),H^1(\Omega)]_\mu} \lesssim \| u_h \|_{[L^2(\Omega),H^1(\Omega)]_\theta}^{1-s}\| u_h \|_{H^1(\Omega)}^s
    \lesssim \| u_h \|_{[L^2(\Omega),H^1(\Omega)]_\theta} \frac{h^{\theta-\mu}}{p^{2(\theta-\mu)}}.
  \end{align*}
\end{proof}
\section{Tools}
In the present section, we collect different tools from the literature and our previous works.
\subsection{Interpolation spaces} 
The first result of the following proposition can be found in~\cite[Thm.~6]{BenBelgacem_94}, while the second
one is a consequence of results available in~\cite{BernardiM_97}, cf. also~\cite[Lemma~5.6]{KMR_19}.
\begin{proposition}\label{prop:GL}
  \begin{enumerate}[(i)]
  \item \label{item:prop:GL-i} There holds for $\theta \in (0,1)$ and $p\in\N_0$
    \begin{align*}
      [ ( \cQ^p,\|\cdot\|_{L^2(\Sref)}),(\cQ^p,\|\cdot\|_{H^1(\Sref)})]_{\theta} 
      = (\cQ^p,\|\cdot\|_{[L^2(\Sref),H^1(\Sref)]_\theta})
      \qquad \mbox{(equivalent norms)}.
    \end{align*}
    and
    \begin{align*}
      [(\wilde\cQ^p,\|\cdot\|_{L^2(\Sref)}),(\wilde\cQ^p,\|\cdot\|_{\wilde H^1(\Sref)})]_{\theta} 
      = (\wilde\cQ^p,\|\cdot\|_{[L^2(\Sref),\wilde H^1(\Sref)]_\theta})
      \qquad \mbox{(equivalent norms)}.
    \end{align*}
  \item \label{item:prop:GL-ii} Let $i_p: C(\overline{\Sref}) \rightarrow \cQ^p$
    be the tensor-product Gau{\ss}-Lobatto interpolation
    operator. Then for every $\theta \in [0,1]$ there exists $C > 0$ such that for all $p$, $q \in \N_0$ the following
    stability estimate holds for the operator $i_p$:
    $$
    \|i_p\|_{(\cQ^p,\|\cdot\|_{H^\theta(\Sref)})  \leftarrow (\cQ^q,\|\cdot\|_{H^\theta(\Sref)})}
    \leq C ( 1 + q/(p+1))^{2-\theta}
    $$
  \end{enumerate}
\end{proposition}
Due to their scaling properties, it is often preferable to work with seminorms instead of full norms.
When working with interpolation spaces, the interpolation between seminorms inherits these advantageous properties
as is stated in the following proposition, taken from~\cite[Lemma~4.1]{KMR_19}.
\begin{proposition}[{\cite[Lem.~4.1]{KMR_19}}]\label{lemma:K-vs-k}
  Let $X_1 \subseteq X_0$ be two continuously embedded Banach spaces with norms 
  $\|\cdot\|_0$ and $\|\cdot\|_1 := H^{-1} \|\cdot\|_0 + |\cdot |_1$, where $|\cdot|_1$ is a seminorm
  and $H > 0$. Introduce the following two $K$-functionals: 
  \begin{align*}
    K(u,t)^2:= \inf_{v \in X_1} \|u - v\|_{0}^2 + t \|v\|_1^2,
    \qquad k(u,t)^2:= \inf_{v \in X_1} \|u - v\|_{0}^2 + t|v|_1^2.
  \end{align*}
  For $\theta \in (0,1)$ introduce the seminorm $|\cdot|_\theta$ and the norms 
  $\|\cdot\|_\theta$ and $\|\cdot\|_{\tilde\theta}$ by
  \begin{align*}
    |u|^2_\theta &= \int_{t=0}^\infty t^{-2\theta} k^2(u,t) \frac{dt}{t}, \\
    \|u\|^2_\theta &= \int_{t=0}^\infty t^{-2\theta} K^2(u,t) \frac{dt}{t}, \\
    \|u\|^2_{\tilde \theta} &= H^{-2\theta}  \|u\|^2_0 + |u|^2_\theta. 
  \end{align*}
  Then there exists $C>0$, which depends solely on $\theta$ (in particular, it is independent of $H$)
  such that
  \begin{align*}
    C^{-1} \|u\|_{\theta} \leq \|u\|_{\tilde \theta} \leq C \|u\|_{\theta}. 
  \end{align*}
\end{proposition}
In the next result we construct a function realizing a quasi-optimal decomposition as required in the $K$-functional for the pair
$L^2(\Omega)$ and $H^1(\Omega)$, with additional local stability properties.
\begin{lemma}\label{lem:K:realization}
Let $\omega\subset\R^2$ be a bounded Lipschitz domain, $\theta\in(0,1)$, and $C>0$. For $u\in [L^2(\omega),H^1(\omega) ]_\theta$, there 
  is a function $w:(0,\infty)\rightarrow H^1(\omega)$ such that
  \begin{align}\label{lem:K:realization:eq1}
    \int_0^\infty t^{-2\theta}\left( \| u - w(t) \|_{L^2(\omega)}^2 + t^2 \| w(t) \|_{H^1(\omega)}^2 \right) \frac{dt}t \lesssim
    \| u \|_{[L^2(\omega),H^1(\omega) ]_\theta}^2.
  \end{align}
  Additionally, if $\omega'\subset\omega$ with $\dist(\omega',\partial\omega)>C$
  then
  \begin{align}\label{lem:K:realization:eq2}
    \| w(t) \|_{L^2(\omega')} \lesssim \| u \|_{L^2(\left\{ x\in\omega \mid d_{\omega'}(x)< t \right\})}.
  \end{align}
\end{lemma}
\begin{proof}
  Let $\rho_t(\cdot) = t^{-2}\rho(\cdot/t)$, where $\rho$ is a mollifier, i.e., a smooth function supported in the unit ball and integrating to $1$.
  It is known that, cf.~\cite[Sec.~2.5, Lem.~10]{Burenkov_98},
  \begin{align*}
    \| \rho_t \star f \|_{H^k(S)} &\lesssim t^{\ell-k}\| f \|_{H^\ell(\left\{ x\in\R^2 \mid d_S(x)<t \right\})},
    \qquad k,\ell\in\left\{ 0,1 \right\}, k\geq\ell,\\
    \| f - \rho_t\star f \|_{L^2(S)} &\lesssim t\| f \|_{H^1(\left\{ x\in\R^2 \mid d_S(x)< t \right\})},
\end{align*}
  for any measurable open set $S\subset\R^2$ with implied constants not depending on $f$ or $S$.
  Employing Stein's extension operator $E$, we will show that the regularized function
  \begin{align*}
    w(t) := \chi_{[0,C/4]}(t)\cdot\rho_t\star Eu,
  \end{align*}
  $\chi_{[a,b]}:\R\rightarrow\R$ being the characteristic function of the interval $[a,b]$,
  fulfills
  \begin{align*}
    \int_0^\infty t^{-2\theta}\left( \| Eu - w(t) \|_{L^2(\R^2)}^2 + t^2 \| w(t) \|_{H^1(\R^2)}^2 \right) \frac{dt}t \lesssim
    \| Eu \|_{[L^2(\R^2),H^1(\R^2) ]_\theta}^2.
  \end{align*}
  The boundedness properties of the extension operator $E$ then imply \eqref{lem:K:realization:eq1}, while the local
  boundedness properties \eqref{lem:K:realization:eq2}
  follow in conjunction with the properties of the convolution with $\rho_t$ given above.
  For the rest of this proof, $K$ denotes the K-functional corresponding to the pair $L^2(\R^2)$ and $H^1(\R^2)$.
  Let $t\in (0,\infty)$ be fixed. Choose a function $v \in H^1(\R^2)$ such that
  \begin{align*}
    \| Eu - v \|_{L^2(\R^2)}^2 + t^2 \| v \|_{H^1(\R^2)}^2 \leq 2 K(t,Eu)^2.
  \end{align*}
  Note that
  \begin{align*}
    \| v - \chi_{[0,C/4]}(t)\cdot\rho_t\star v \|_{L^2(\R^2)} \lesssim
    \begin{cases}
      \| v \|_{L^2(\R^2)} \leq \frac{4}{C} t \| v \|_{H^1(\R^2)} \quad \text{ for } t>\frac{C}4\\
      t \| v \|_{H^1(\R^2)} \quad \text{ for } t\leq \frac{C}{4}.
    \end{cases}
  \end{align*}
  Then,
  \begin{align*}
    \| Eu - \chi_{[0,C/4]}(t)\cdot\rho_t\star v \|_{L^2(\R^2)} + &t \|\chi_{[0,C/4]}(t)\cdot\rho_t\star v \|_{H^1(\R^2)}\\
    &\lesssim \| Eu - v \|_{L^2(\R^2)} + \| v - \chi_{[0,C/4]}(t)\cdot\rho_t\star v \|_{L^2(\R^2)}
    + t\| v \|_{H^1(\R^2)}\\
    &\lesssim \| Eu - v \|_{L^2(\R^2)} + t \| v \|_{H^1(\R^2)}.
  \end{align*}
  Hence,
  \begin{align*}
    \| Eu - \chi_{[0,C/4]}(t)\cdot\rho_t\star Eu \|_{L^2(\R^2)} &\leq \| Eu - \chi_{[0,C/4]}(t)\cdot\rho_t\star v \|_{L^2(\R^2)} + \| \chi_{[0,C/4]}(t)\cdot\rho_t\star(v-Eu) \|_{L^2(\R^2)}\\
    &\lesssim \| Eu - \chi_{[0,C/4]}(t)\cdot\rho_t\star v \|_{L^2(\R^2)} + \| v - Eu \|_{L^2(\R^2)}\\
    &\lesssim \| Eu - v \|_{L^2(\R^2)} + t \| v \|_{H^1(\R^2)}
  \end{align*}
  as well as
  \begin{align*}
    \| \chi_{[0,C/4]}(t)\cdot\rho_t\star Eu \|_{H^1(\R^2)} &\leq \| \chi_{[0,C/4]}(t)\cdot\rho_t\star v \|_{H^1(\R^2)} + \| \chi_{[0,C/4]}(t)\cdot\rho_t\star(v-Eu) \|_{H^1(\R^2)}\\
    &\lesssim \| v \|_{H^1(\R^2)} + t^{-1} \| v-Eu \|_{L^2(\R^2)}.
  \end{align*}
  Hence,
  \begin{align*}
    \| Eu - \chi_{[0,C/4]}(t)\cdot\rho_t\star Eu \|_{L^2(\R^2)}^2 + t^2 \| \chi_{[0,C/4]}(t)\cdot\rho_t\star Eu \|_{H^1(\R^2)}^2 \lesssim K(t,Eu)^2.
  \end{align*}
  Multiplying this last estimate by $t^{-2\theta-1}$ and integrating in $t$ shows the result.
\end{proof}
The next result states certain equivalences between fractional order norms defined by interpolation
and by double integrals, on spaces including partial homogeneous boundary conditions.
\begin{lemma}\label{lem:equiv:lm:inter} 
  \begin{enumerate}
      Let $\theta\in (0,1)$.
    \item[(i)]
      Let $\what\cE\subset\left\{ \what e_4, \what e_5, \what e_6 \right\}$ be a subset of
      the edges of $\Tref$. Then, there holds
      \begin{align*}
        | v |_{H^\theta(\Tref)} + \| d_{\what\cE}^{-\theta} v \|_{L^2(\Tref)} \lesssim
        \| v \|_{[L^2(\Tref), \wilde H^1_{\what\cE}(\Tref)]_\theta}.
      \end{align*}
    \item[(ii)] There holds
      \begin{align*}
	\| v \|_{[L^2(\Sref),\wilde H^1(\Sref)]_\theta} \lesssim
	| v |_{H^\theta(\Sref)} + \| d_{\partial\Sref}^{-\theta} v \|_{L^2(\Sref)}.
      \end{align*}
  \end{enumerate}
\end{lemma}
\begin{proof}
\begin{figure}[htb]
  \centering
  \psfrag{S}{$\Sref$}
  \psfrag{w}{$\omega$}
  \includegraphics[width=0.8\textwidth]{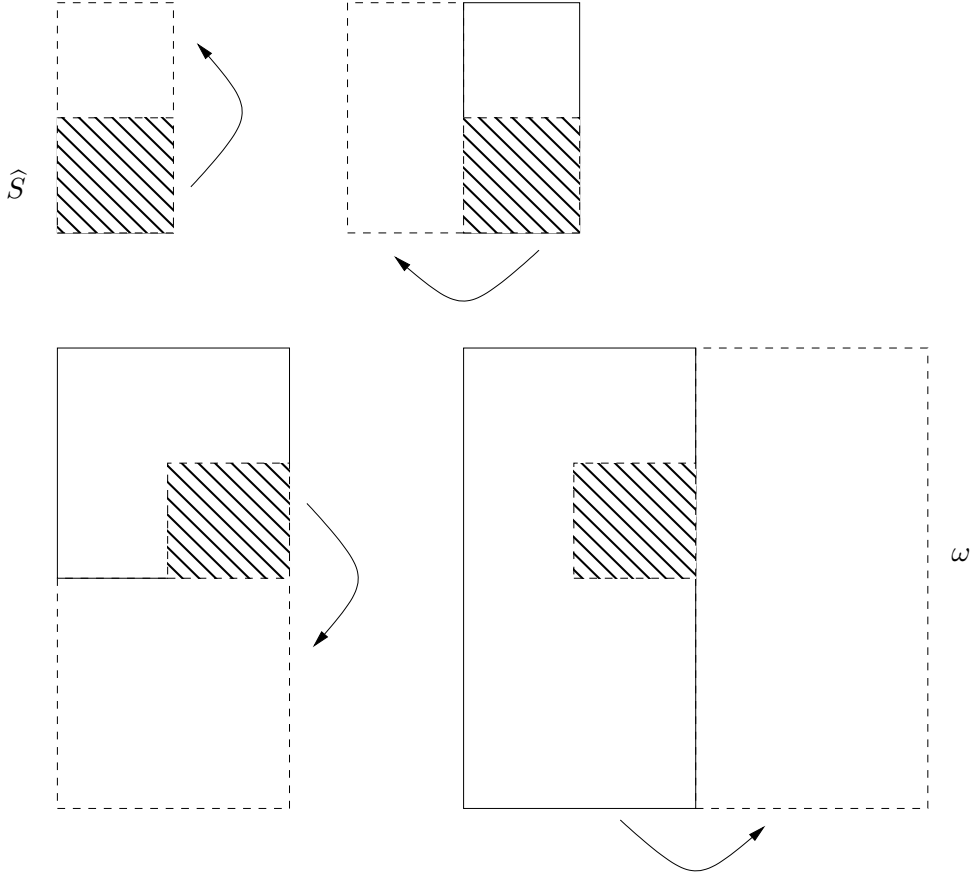}
  \caption{Extension procedure in Lemma~\ref{lem:equiv:lm:inter} for the function $v$ from $\Sref$ to $\omega$.}
  \label{fig:ext}
\end{figure}
  First, note that
  \begin{align*}
    | v |_{H^\theta(\Tref)} &\leq \| v \|_{H^\theta(\Tref)}
    \lesssim \| v \|_{[L^2(\Tref),H^1(\Tref)]_\theta}
    \leq \| v \|_{[L^2(\Tref),\wilde H^1_{\what\cE}(\Tref)]_\theta}.
  \end{align*}
  Here, the first estimate follows by definition, the second from Lemma~\ref{lem:refel:norms}, and
  the last one from $\wilde H^1_{\what\cE}(\Tref)\subset H^1(\Tref)$.
  Next, let $\what e_j\in\what\cE$. For $v|_{\what e_j}=0$, a simple argument based on Hardy's inequality,
  cf.~\cite[Lem.~4.4 (iii)]{KMR_19}, shows that
  \begin{align*}
    \| d_{\what e_j}^{-1} v \|_{L^2(\Tref)} \lesssim \| v \|_{H^1(\Tref)}.
  \end{align*}
  Interpolating this estimate with $\| v \|_{L^2(\Tref)} \leq \| v \|_{L^2(\Tref)}$ shows that
  \begin{align*}
    \| d_{\what e_j}^{-\theta} v \|_{L^2(\Tref)} \lesssim
    \| v \|_{[L^2(\Tref), \wilde H^1_{\what\cE}(\Tref)]_\theta},
  \end{align*}
  and the obvious estimate $d_{\what\cE}^{-1} \leq \sum_{\what e_j\in\what\cE}d_{\what e_j}^{-1}$
  concludes the statement $(i)$.
  To show $(ii)$, we extend the function $v$ given on $\Sref$
  to a function $v$ on a rectangle $\omega$ such that the boundaries of $\Sref$ and $\omega$ have positive distance
  $\dist(\Sref,\partial\omega)>0$.
  This extension is done by mirroring symmetrically along certain edges, cf. Figure~\ref{fig:ext}.
  In particular, as $\| v \|_{L^2(\omega)} \lesssim \| v \|_{L^2(\Sref)}$ and $\| v \|_{H^1(\omega)} \lesssim \| v \|_{H^1(\Sref)}$,
  the interpolation theorem gives
  \begin{align*}
    \| v \|_{[L^2(\omega),H^1(\omega)]_\theta} \lesssim \| v \|_{[L^2(\Sref),H^1(\Sref)]_\theta}.
  \end{align*}
  Next, we apply Lemma~\ref{lem:K:realization} with $C=\dist(\Sref,\partial\omega)$ to obtain a function $w$ such that
  \begin{align*}
    \int_0^\infty t^{-2\theta}
    \left( \| v - w(t) \|_{L^2(\omega)}^2 + t^2 \| w(t) \|_{H^1(\omega)}^2 \right)\frac{dt}{t}
    \lesssim \| v \|_{[L^2(\omega),H^1(\omega)]_\theta}^2.
  \end{align*}
  Note that the preceding two estimates imply
  \begin{align}\label{lem:equiv:lm:inter:eq2}
    \int_0^\infty t^{-2\theta}
    \left( \| v - w(t) \|_{L^2(\Sref)}^2 + t^2 \| w(t) \|_{H^1(\Sref)}^2 \right)\frac{dt}{t}
    \lesssim \| v \|_{[L^2(\Sref),H^1(\Sref)]_\theta}^2.
  \end{align}
  We denote $S_t = \left\{ x\in\Sref| d_{\partial\Sref}(x)<t \right\}$ and see
  \begin{align*}
    \| w(t) \|_{L^2(S_t)} \lesssim \| v \|_{L^2(\left\{ x\in \omega\mid d_{S_t}(x)< t \right\})} \lesssim \| v \|_{L^2(S_{2t})},
  \end{align*}
  where the first estimate follows from Lemma~\ref{lem:K:realization}, and the last
  estimate follows from the local stability of the specific extension of $v$, cf. Figure~\ref{fig:ext}.
  Choose a smooth cut-off function $\chi_t$ with $\| \chi_t \|_{L^\infty}\leq 1$, $\Sref\cap\supp\chi_t\subset S_t$, $\chi_t|_{S_{t/2}}=1$,
  and $\| \nabla \chi_t \|_{L^\infty(\Sref)} \lesssim t^{-1}$.
  Define $\wilde w(t) := (1-\chi_t)w(t)$, note $\wilde w(t)\in \wilde H^1(\Sref)$ as well as
  \begin{align}\label{lem:equiv:lm:inter:eq3}
    \begin{split}
    \| v - \wilde w(t) \|_{L^2(\Sref)} &\leq \| v - w(t) \|_{L^2(\Sref)} + \| w(t) \|_{L^2(S_t)}
    \leq \| v - w(t) \|_{L^2(\Sref)} + \| v \|_{L^2(S_{2t})},\\
    t \| \wilde w(t) \|_{H^1(\Sref)} &\lesssim t \| w(t) \|_{H^1(\Sref)} + \| w(t) \|_{L^2(S_t)}
    \leq t \| w(t) \|_{H^1(\Sref)} + \| v \|_{L^2(S_{2t})}.
    \end{split}
  \end{align}
  The definition of the interpolation norm, the preceding estimates, and~\eqref{lem:equiv:lm:inter:eq2} show
  \begin{align*}
    \| v \|_{[L^2(\Sref),\wilde H^1(\Sref)]_\theta}^2 &\leq
    \int_0^\infty t^{-2\theta}
    \left( \| v - \wilde w(t) \|_{L^2(\Sref)}^2 + t^2 \| \wilde w(t) \|_{H^1(\Sref)}^2 \right)\frac{dt}{t}\\
    &\leq
    \int_0^\infty t^{-2\theta}
    \left( \| v - w(t) \|_{L^2(\Sref)}^2 + t^2 \| w(t) \|_{H^1(\Sref)}^2 \right)\frac{dt}{t}
    + \int_0^\infty t^{-2\theta} \| v \|_{L^2(S_{2t})}^2 \frac{dt}t\\
  &\hspace{-.15cm}\stackrel{\eqref{lem:equiv:lm:inter:eq2}}\lesssim
    \| v \|_{[L^2(\Sref),H^1(\Sref)]_\theta}^2
    + \int_0^\infty t^{-2\theta} \| v \|_{L^2(S_{2t})}^2 \frac{dt}t,
  \end{align*}
  Due to the $L^2$ and $H^1$ stability of the extension of $v$, Lemma~\ref{lem:refel:norms},
  a simple substitution $\tau = 2t$, and the trivial bound $\| v \|_{L^2(\Sref)} \lesssim \| d_{\Sref}^{-\theta}v \|_{L^2(\Sref)}$,
  we conclude for any fixed $\delta>0$
  \begin{align*}
    \| v \|_{[L^2(\Sref),\wilde H^1(\Sref)]_\theta}^2
    &\lesssim | v |_{H^\theta(\Sref)}^2 + \| d_{\partial\Sref}^{-\theta} v \|_{L^2(\Sref)}^2
    + \int_0^\delta \tau^{-2\theta} \| v \|_{L^2(S_{\tau})}^2 \frac{d\tau}\tau.
  \end{align*}
  It remains to bound the last integral on the right-hand side by $\| d_{\partial\Sref}^{-\theta} v \|_{L^2(\Sref)}$.
  To that end, choose $\delta>0$ small enough so that $S_\delta \neq \Sref$ and consider for $\tau\leq\delta$ the
  (overlapping) decomposition $S_\tau = S_\tau^{(1)}\cup S_\tau^{(2)}\cup S_\tau^{(3)}\cup S_\tau^{(4)}$ given by
  \begin{align*}
    S_\tau^{(1)} &= (-1,-1+\tau)\times (-1/\sqrt{3},2/\sqrt{3}), &S_\tau^{(2)} = (1-\tau,1)\times (-1/\sqrt{3},2/\sqrt{3})\\
    S_\tau^{(3)} &= (-1,1) \times (-1/\sqrt{3},-1/\sqrt{3}+\tau), &S_\tau^{(4)} = (-1,1) \times (2/\sqrt{3}-\tau,2/\sqrt{3})
  \end{align*}
  and note that
  \begin{align*}
    \int_0^\delta \tau^{-2\theta} \| v \|_{L^2(S_{\tau})}^2 \frac{d\tau}\tau \leq\sum_{j=1}^4 \int_0^\delta \tau^{-2\theta} \| v \|_{L^2(S_{\tau}^{(j)})}^2 \frac{d\tau}\tau.
  \end{align*}
  We will bound only the right-hand side term for $j=1$, the three remaining terms can be bounded similarly. Using Fubini, we see
  \begin{align*}
    \int_0^\delta \tau^{-2\theta} \| v \|_{L^2(S_{\tau}^{(1)})}^2 \frac{d\tau}\tau
    &= \int_0^\delta \tau^{-2\theta-1} \int_{0}^{\tau} \int_{-1/\sqrt{3}}^{2/\sqrt{3}} |v(-1+x,y)|^2\,dy\,dx\,d\tau\\
    &= \int_{0}^\delta \int_{x}^\delta \tau^{-2\theta-1} \int_{-1/\sqrt{3}}^{2/\sqrt{3}} |v(-1+x,y)|^2\,dy\,d\tau\,dx\\
    &\lesssim \int_{0}^\delta x^{-2\theta} \int_{-1/\sqrt{3}}^{2/\sqrt{3}} |v(-1+x,y)|^2\,dy\,dx.
  \end{align*}
  We conclude the proof by noting that for $(x,y)\in S_\delta$ there holds $x^{-\theta} \leq \dist_{\partial\Sref}^{-\theta}(x,y)$.
\end{proof}
\subsection{Decomposition of FEM spaces} 
The following result summarizes the main results of~\cite{KMR_19}.
\begin{lemma}\label{lem:kmr:decomp:2}
  Let $\cT$ be a mesh of $\Omega$ that fulfills Assumption~\ref{assumption:refmaps}
  and $\bp = \left( p_K \right)_{K\in\cT}$ be a degree distribution on $\cT$ that
  fulfills Assumption~\ref{assumption1}.
  For all vertices $V\in\cV^{\rm int}$ and edges $e\in\cE^{\rm int}$ there
  exist
  \begin{enumerate}
    \item[(i)]\label{lem:kmr:decomp:2:iii}
      polynomial degrees $q_V$ and $q_e$ and associated spaces
      \begin{align*}
	X_{\omega_V} &:= \{ \wilde u \in \cP^{q_V}(\Tref) \mid \wilde u|_{\widehat e_6}=0, \wilde u \text{ symmetric w.r.t.
	    the line from } \widehat\bv_3 \text{ to the origin} \},\\
	X_{\omega_e} &:= \{ \wilde u \in \cP^{q_e}(\Tref) \mid \wilde u|_{\what e_5\cup\what e_6}=0\};
      \end{align*}
    \item[(ii)]\label{lem:kmr:decomp:2:i}
      push-forward operators $T_{\omega}: X_\omega \rightarrow \wilde\cS^{\bp,1}(\cT|_{\omega})$,
      $\omega\in\left\{ \omega_V \mid V\in\cV^{\rm int} \right\}\cup\left\{ \omega_e \mid e \in \cE^{\rm int} \right\}$,
      such that for any $\varepsilon>0$ sufficiently small there hold the mapping properties
      \begin{align*}
	h_{\omega}^{-1} \| T_\omega\wilde u \|_{L^2(\omega)} &\lesssim
	\| \wilde u \|_{L_2(\Tref)} + \| \wilde u \|_{L^{\infty}(\Tref_\varepsilon)},\\
	\| \nabla T_\omega\wilde u \|_{L^2(\omega)} &\lesssim 
	\| \nabla \wilde u \|_{L_2(\Tref)} + \| \nabla \wilde u \|_{L^{\infty}(\Tref_\varepsilon)},
      \end{align*}
      for all polynomials $\wilde u\in X_\omega$, with constants depending only on $\varepsilon$, and
  \end{enumerate}
  such that every function $u\in\wilde\cS^{\bp,1}(\cT)$ can be written as
  \begin{align}\label{lem:kmr:decomp:2:ii}
    u = u_1 + \sum_{V\in\cV^{\rm int}} T_{\omega_V}(\wilde u_V) + \sum_{e\in\cE^{\rm int}}T_{\omega_e}(\wilde u_e) + \sum_{K\in\cT} u_K,
  \end{align}
  with functions $u_1\in\wilde\cS^{1,1}(\cT)$, $\wilde u_V\in X_{\omega_V}$,
  $\wilde u_e\in X_{\omega_e}$, and $u_K\in \cP^{p_K}\cap\wilde H^1(K)$, all of them depending linearly on $u$, and such that
  for $\theta\in(0,1)$ and $\delta>0$ sufficiently small there holds
  \begin{align*}
    \| u_1 \|_{[L^2(\Omega),\wilde H^1(\Omega)]_\theta}^2 &\lesssim \| u \|_{[L^2(\Omega),\wilde H^1(\Omega)]_\theta}^2,\\
    \sum_{V\in\cV^{\rm int}} h_{\omega_V}^{2-2\theta} \left( |\wilde u_V|_{H^\theta(\Tref)}^2
    + \| d_{\widehat e_6}^{-\theta} \wilde u_V \|_{L^2(\Tref)}^2
    + \| \wilde u_V \|_{W^{1,\infty}(\Tref_\delta)}^2 \right)
    &\lesssim \| u \|_{[L^2(\Omega),\wilde H^1(\Omega)]_\theta}^2,\\
    \sum_{e\in\cE^{\rm int}} h_{\omega_e}^{2-2\theta} \left( |\wilde u_e|_{H^\theta(\Tref)}^2
    + \| d_{\widehat e_5\cup \widehat e_6}^{-\theta} \wilde u_e \|_{L^2(\Tref)}^2
    + \| \wilde u_e \|_{W^{1,\infty}(\Tref_\delta)}^2 \right)
    &\lesssim \| u \|_{[L^2(\Omega),\wilde H^1(\Omega)]_\theta}^2,\\
    \sum_{K\in\cT} | u_K |_{H^\theta(K)}^2 + \| d_{\partial K}^{-\theta} u_K \|_{L^2(K)}^2
    &\lesssim \| u \|_{[L^2(\Omega),\wilde H^1(\Omega)]_\theta}^2.
  \end{align*}
  The implied constants depend only on $\theta$ and $\delta$.
  The result remains true if we use $\cS^{\bp,1}(\cT)$ instead of $\wilde\cS^{\bp,1}(\cT)$
  and $H^1(\Omega)$ instead of $\wilde H^1(\Omega)$. In this case, all vertices $V\in\cV$ and edges $e\in\cE$
  have to be taken into account.
\end{lemma}
\begin{proof}
  Depending on the type of the underlying patch, the operators $T_{\omega}$ are defined as either
  \begin{align*}
    (T_{\omega_V} \wilde u)|_{K'} :=
    \begin{cases}
      \wilde u\circ F_{K'}^{-1} & \text{ if } \what{K'}=\Tref,\\
      (\cD \wilde u)\circ F_{K'}^{-1} &\text{ if } \what{K'}=\Sref,
    \end{cases}
  \end{align*}
  or
  \begin{align*}
    (T_{\omega_e} \wilde u)|_{K'} :=
    \begin{cases}
      \wilde u\circ F_{K'}^{-1} & \text{ if } \what{K'}=\Tref,\\
      (i_{\lfloor q_e/2\rfloor}\cD \wilde u)\circ F_{K'}^{-1} &\text{ if } \what{K'}=\Sref,
    \end{cases}
  \end{align*}
  and taking into account certain issues of orientation of the element maps $F_K$. Here,
  $i_p: C(\overline{\Sref}) \rightarrow {\mathcal Q}_p$ denotes the Gauss-Lobatto interpolation operator. These operators fulfill
  $(ii)$, which follows by scaling arguments, Lemma~\ref{lem:2dduffy}, and a possible application
  of Proposition~\ref{prop:GL}, $(ii)$.
  Next,~\cite[Thm.~2.6]{KMR_19} states that
  \begin{align*}
    u = u_1 + \sum_{V\in\cV^{\rm int}}  u_V
    + \sum_{e\in\cE^{\rm int}}u_e + \sum_{K\in\cT} u_K,
  \end{align*}
  where $u_1=u-I_h u$ with $I_h$ being the Scott-Zhang projection operator.
  The proofs of~\cite[Thm.~2.6, Cor.~6.2, Lem.~6.1]{KMR_19} reveal that
  $u_V = T_{\omega_V}(\wilde u_V)$ for some function $\wilde u_V\in X_V$, and that
  \begin{align*}
    \sum_{V\in\cV^{\rm int}} 
    h_{\omega_V}^{2-2\theta} \| \wilde u_V \|_{[L^2(\Tref),\wilde H^1_{\what e_6}(\Tref)]_\theta}^2
    + \| \wilde u_V \|_{W^{1,\infty}(\Tref_\delta)}^2
    \lesssim \| u \|_{[L^2(\Omega),\wilde H^1(\Omega)]_\theta}^2.
  \end{align*}
  The proofs of~\cite[Thm.~2.6, Cor.~6.4, Lem.~6.3]{KMR_19}
  reveal that $u_e = T_{\omega_e}(\wilde u_e)$ for some function $\wilde u_e\in X_e$, and that
  \begin{align*}
    \sum_{e\in\cE^{\rm int}} 
    h_{\omega_e}^{2-2\theta} \| \wilde u_e \|_{[L^2(\Tref),\wilde H^1_{\what e_5\cup \what e_6}(\Tref)]_\theta}^2
    + \| \wilde u_e \|_{W^{1,\infty}(\Tref_\delta)}^2
    \lesssim \| u \|_{[L^2(\Omega),\wilde H^1(\Omega)]_\theta}^2.
  \end{align*}
  The proof of~\cite[Thm.~2.6]{KMR_19} reveals
  \begin{align*}
    \sum_{K\in\cT} | u_K |_{H^\theta(K)}^2 + \| d_{\partial K}^{-\theta} u_K \|_{L^2(K)}^2
    &\lesssim \| u \|_{[L^2(\Omega),\wilde H^1(\Omega)]_\theta}^2.
  \end{align*}
  Finally, Lemma~\ref{lem:equiv:lm:inter}, $(i)$, shows the stipulated estimates.
\end{proof}
\section{Proofs of the main results}
\subsection{Interpolation spaces as trace spaces} 

We recall that \cite[Lemma~{40.1}]{Tartar} states the following: 
$u \in [X_0,X_1]_{\theta}$ if and only if there exists a function $v:(0,\infty) \rightarrow X_1$ 
with $t^{1-\theta} \|v\|_{1} \in L^2(\R_+,\frac{dt}{t})$  whose derivative $v^\prime$ exists and satisfies 
$t^{1-\theta} \|v^\prime\|_{0} \in L^2(\R_+,\frac{dt}{t})$ and $\lim_{t \rightarrow 0} v(t) = v(0) = u$ (in $X_0$). 
The following lemma shows that a similar characterization can be achieved for the $|\cdot|_\theta$-seminorm.
\begin{lemma}\label{lemma:k-trace-space}
  Under the hypotheses of Proposition~\ref{lemma:K-vs-k} there holds the following for $\theta\in(0,1)$:
  \begin{enumerate}[(i)]
    \item\label{item:lemma:k-trace-space-i}
      Let $u \in [X_0,X_1]_{\theta}$. Then, there exists a function $v:(0,\infty) \rightarrow X_1$
      with $t^{1-\theta} |v|_{1} \in L^2(\R_+,\frac{dt}{t})$ whose derivative $v^\prime$ satisfies 
      $t^{1-\theta} \|v^\prime\|_{0} \in L^2(\R_+,\frac{dt}{t})$ and $\lim_{t\rightarrow 0} v(t) = u$ (convergence in $X_0$). 
      Moreover, for a $C>0$ that depends solely on $\theta$
      \begin{align*}
	\int_{t=0}^\infty t^{2(1-\theta)} |v(t)|^2_1 \frac{dt}{t} + 
	\int_{t=0}^\infty t^{2(1-\theta)} \|v^\prime(t)\|^2_0 \frac{dt}{t} 
	\leq C |u|^2_\theta.
      \end{align*}
    \item\label{item:lemma:k-trace-space-ii}
      Let $v:(0,\infty) \rightarrow X_1$ be such that $t^{1-\theta} |v|_{1} \in L^2(\R_+,\frac{dt}{t})$ and 
      $t^{1-\theta} \|v^\prime\|_{0} \in L^2(\R_+,\frac{dt}{t})$. Then $\lim_{t \rightarrow 0} v(t)$ exists (in $X_0$) 
      and $v(0) \in [X_0,X_1]_{\theta}$. Moreover, for a $C>0$ depending solely on $\theta$
      \begin{align*}
	|v(0)|^2_\theta  \leq 
	\int_{t=0}^\infty t^{2(1-\theta)} |v(t)|^2_1 \frac{dt}{t} + 
	\int_{t=0}^\infty t^{2(1-\theta)} \|v^\prime(t)\|^2_0 \frac{dt}{t}.  
      \end{align*}
  \end{enumerate}
\end{lemma}
\begin{proof}
  We modify the arguments presented in~\cite[Lem.~40.1]{Tartar}.\\

  \textbf{Proof of~(\ref{item:lemma:k-trace-space-i}).} Due to Proposition~\ref{lemma:K-vs-k} we know that $|u|_\theta<\infty$.
  For every $n \in \Z$ pick $v_n \in X_1$ such that 
  \begin{equation}
  \label{eq:lemma:k-trace-space-10}
  \|u - v_n\|_{0} + e^n |v_n|_{1} \leq 2 k(u,e^n) 
  \end{equation}
  and define on $(0,\infty)$ the function $v$ as the piecewise linear interpolant with values 
  $v_n$ at the knots $t_n = e^n$. Note that for $n\rightarrow -\infty$ there holds $k(u,e^n)\rightarrow 0$, and hence also
  $v_n \rightarrow u$ in $X_0$. We conclude that $\lim_{t\rightarrow 0}v(t)=u$ in $X_0$.
  Next, in view of 
  \begin{align}\label{eq:lemma:k-trace-space-15}
    k(u,\lambda t) \leq \max\{1,\lambda \} k(u,t),\qquad   \lambda > 0,
  \end{align}
  we have for $t \in (e^n,e^{n+1})$
  \begin{align}\label{eq:lemma:k-trace-space-20}
    |v(t)|_1 \leq \max\{|v_n|_1,|v_{n+1}|_1\} \leq 2 \max\{ e^{-n} k(u,e^n), e^{-(n+1)} k(u,e^{n+1})\}
    \leq 2 e^{-n}  k(u,e^n). 
  \end{align}
  Therefore, by exploiting (\ref{eq:lemma:k-trace-space-15}) we get 
  \begin{align}
    \int_{t=0}^\infty t^{2(1-\theta)} |v(t)|_{1}^2 \frac{dt}{t} \lesssim  \sum_{n \in \Z} e^{2(1-\theta) n} e^{-2n} k^2(u,e^n) 
    \sim \int_{t=0}^\infty t^{-2\theta} k^2(u,t)\frac{dt}{t} = |u|^2_\theta. 
  \end{align}
  For $v^\prime$, we note that on the interval $(e^n,e^{n+1})$ we have 
  $v^\prime(t) = \frac{v_{n+1} - v_n}{e^{n+1} - e^n}$ and therefore for $t \in (e^n,e^{n+1})$
  \begin{align*}
    \|v^\prime(t)\|_{0} &\leq \frac{1}{e-1} e^{-n} (\|v_{n+1} - u\|_{0} + \|v_n - u\|_{0}) 
    \leq \frac{2}{e-1} e^{-n} (k(u,e^{n+1}) + k(u,e^n))  \\
    &\leq \frac{2(1+e)}{e-1} e^{-n} k(u,e^{n}), 
  \end{align*}
  so that 
  \begin{align*}
    \int_{t=0}^\infty t^{2(1-\theta)} \|v^\prime(t)\|^2_0 \frac{dt}{t} 
    \lesssim \sum_{n \in \Z} e^{2(1-\theta) n} e^{-2n} k^2(u,e^n)
    \sim \int_{t=0}^\infty t^{-2\theta} k^2(u,t)\frac{dt}{t} = |u|^2_\theta.
  \end{align*}
  \textbf{Proof of~(\ref{item:lemma:k-trace-space-ii}).}
  From $v(t) - v(\varepsilon) = \int_{\varepsilon}^t v^\prime(s)\,ds$, we infer 
  for $ t \ge \varepsilon$ the estimate
  \begin{align*}
    \|v(t) - v(\varepsilon)\|_{0} \leq 
    \sqrt{\int_{s=0}^t s^{2(1-\theta)} \|v^\prime(s)\|^2_{0}\, \frac{ds}{s} }
    \sqrt{\int_{\varepsilon}^t s^{-1+2\theta}\,ds} 
    \leq C Z(v,t) t^{\theta}, 
  \end{align*}
  where 
  $\displaystyle Z(v,t)^2:=\int_{s=0}^t s^{2(1-\theta)} \|v^\prime(s)\|^2_{0}\, \frac{ds}{s}$.  
  By assumption, $\sup_{t>0} Z(v,t) < \infty$. This shows that $\lim_{t \rightarrow 0} v(t) =: v(0)$ exists 
  (convergence in $X_0$). 
  Next, we estimate 
  $k(v(0),t)  = \inf_{w \in X_1} \|v(0) - w\|_{0} + t|w|_{1} \leq \|v(0) - v(t)\|_{0} + t|v(t)|_{1}$.
  The observation $v(t) - v(0) =\int_{s=0}^t v^\prime(s),ds$ implies with Hardy's inequality 
  \begin{align*}
    \int_{t=0}^\infty t^{-2\theta} \|v(t) - v(0)\|_0^2\frac{dt}{t} 
    \leq 
    \int_{t=0}^\infty t^{-2\theta+2} \left|\frac{1}{t} \int_{s=0}^t \|v^\prime(s)\|_{0} \,ds\right|^2\frac{dt}{t}
    \lesssim 
    \int_{t=0}^\infty t^{2 (1-\theta)}  \|v^\prime(t)\|^2_{0} \,\frac{dt}{t} 
  \end{align*}
  so that  
  \begin{align*}
    |v(0)|^2_\theta &= \int_{t=0}^\infty t^{-2\theta} k^2(v(0),t)\frac{dt}{t} 
    \leq \int_{t=0}^\infty t^{-2\theta} \|v(t) - v(0)\|_{0}^2  + t^{2-2\theta} |v(t)|_1^2\frac{dt}{t}\\
    &\lesssim  \int_{t=0}^\infty t^{2(1-\theta)} \|v^\prime\|_{0}^2  + t^{2-2\theta} |v(t)|_1^2\frac{dt}{t}. 
  \end{align*}
  This concludes the argument.
\end{proof}
\subsection{Liftings from a triangle to a tetrahedron and prism} 
As laid out in Section~\ref{sec:construction}, we present now the lifting operator $\cA$ from the reference triangle to the reference
tetrahedron. This operator is the first building block of our overall lifting procedure.
\begin{lemma}\label{lemma:lifting-from-triangle}
  There exists a linear operator $\cA : L^1_{loc}(\Tref) \rightarrow C^\infty(\tetref)$ with the following properties:
  \begin{enumerate}[(i)]
    \item \label{item:lemma:lifting-from-triangle-1} If $u$ is a polynomial of degree $p\geq 0$, then $\cA $ is a polyonomial of degree $p$.
    \item \label{item:lemma:lifting-from-triangle-2} If $u$ is continuous at a point $\bx\in\Tref$, then $(\cA u)(\bx,0) = u(\bx)$.
    \item \label{item:lemma:lifting-from-triangle-3} 
      For every $\gamma > -1/2$ there is a constant $C_\gamma$ such that
      \begin{align*}
	\| d_{\Tref\times\{0\} }^\gamma \cA u \|_{L^2(\tetref)} \leq C_\gamma \| u \|_{L^2(\Tref)}.
      \end{align*}
    \item \label{item:lemma:lifting-from-triangle-4} Let $\widehat f\neq \Tref\times\{0\}$ be a face of $\tetref$ and
      $\widehat e$ be the edge shared by $\widehat f$ and $\Tref\times\{0\}$. Then, for every $\gamma\in\R$ there is a constant $C_\gamma>0$ such that
      \begin{align*}
	\| d_{\widehat e}^\gamma \cA  u \|_{L^2(\widehat f)} \leq C_\gamma \| d_{\widehat e}^\gamma u \|_{L^2(\Tref)}.
      \end{align*}
    \item \label{item:lemma:lifting-from-triangle-5} Let $\widehat e$ be an edge of $\tetref$ from the top $\left( 0,0,1 \right)$ to the vertex
      $\widehat V\neq\left( 0,0,1 \right)$. Then, for every $\gamma<3/2$ there is a constant $C_\gamma>0$ such that
      \begin{align*}
	\| d_{\widehat V}^\gamma \cA u \|_{L^2(\widehat e)} \leq C_{\gamma}\| d_{\widehat V}^{\gamma - 1/2} u \|_{L^2(\Tref)}.
      \end{align*}
    \item \label{item:lemma:lifting-from-triangle-6}
      Let $\widehat f\neq \Tref\times\{0\}$ be a face of $\tetref$ and
      $\widehat e$ be an edge of $\tetref$ from the top $\left( 0,0,1 \right)$ to the vertex $\widehat V\neq\left( 0,0,1 \right)$.
      For every $\theta\in(0,1)$ and $k\in\N$ there is a constant $C_{\theta,k}>0$ such that
      for $\bk\in\N_0^3$ with $\abs{\bk}=k\geq 1$ there holds
      \begin{align*}
	\| d_{\Tref\times\{0\}}^{k-1/2-\theta} \partial^{\bk} \cA u \|_{L^2(\tetref)} + 
	\| d_{\Tref\times\{0\}}^{k-\theta} \partial^{\bk} \cA u \|_{L^2(\widehat f)} + 
	\| d_{\widehat V}^{k+1/2-\theta} \partial^{\bk} \cA u \|_{L^2(\widehat e)}
	\leq C_{\theta,k}| u |_{H^\theta(\Tref)}.
      \end{align*}
    \item \label{item:lemma:lifting-from-triangle-7} For every $\varepsilon>0$ and $j\in\N\cup\left\{ 0 \right\}$, there is $C_{\varepsilon,j}>0$
      such that
      \begin{align*}
	\| \cA  u \|_{W^{j,\infty}(\tetref_\varepsilon)} \leq C_{\varepsilon,j} \| u \|_{L^2(\Tref)}.
      \end{align*}
    \item \label{item:lemma:lifting-from-triangle-8} For $\varepsilon>0$ consider the set
      $\Tref_{\varepsilon}$,
      and for $z\in[0,1]$ the scaled versions $(1-z)\Tref_\varepsilon$.
      Then, for arbitrary $\delta$ with $\delta>\varepsilon$
      and $\bk\in\N_0^3$ with $\abs{\bk}\leq 1$ there holds
      \begin{align*}
	\| \partial^\bk \cA u (\cdot,z) \|_{L^{\infty}( (1-z)\Tref_\varepsilon)}
	\lesssim \| u \|_{L^2(\Tref)} + \| u \|_{W^{\abs{\bk},\infty}(\Tref_\delta)},
      \end{align*}
      with a constant depending only on the difference $\delta-\varepsilon$.
  \end{enumerate}
\end{lemma}
\begin{proof}
  The lifting operator $\cA $ will be defined by an averaging process. To that end, define for $(\bx,z)\in\tetref$ the mapping
  \begin{align*}
    F_{(\bx,z)}:
    \begin{cases}
      \Tref \rightarrow \Tref\\
      \bxi \mapsto \bx + \frac{z}{2} \bxi.
    \end{cases}
  \end{align*}
  Fix a mollifier $\rho\in C^{\infty}(\R^2)$ with $\supp(\rho)\subset \Tref$ and $\int_{\Tref}\rho(\by)\,d\by=1$ and define
  \begin{align}\label{eq:lemma:lifting-from-triangle:0}
    ( \cA  u )(\bx,z) := \int_{\Tref} \rho(\bxi) u ( F_{(\bx,z)}(\bxi))\,d\bxi
    = \frac{4}{z^2} \int_{\R^2} \rho\left( \frac{\bs-\bx}{z/2} \right) u(\bs)\,d\bs.
  \end{align}
  This formula is well defined as $(\bx,z)\in\tetref$ implies $F_{(\bx,z)}(\Tref)\subset\Tref$.
  We note that $\cA u\in C^{\infty}(\tetref)$.
  We calculate for $j=1,2$, cf.~\cite[Lemma 1.4.1.4]{Grisvard}, for any constant $c\in\R$
  \begin{align*}
    (\partial_j\cA u)(\bx,z) &= \frac{2}{z} \int_{\Tref} \partial_j\rho(\bxi) [c - u(F_{(\bx,z)}(\bxi))]\,d\bxi,\\
    (\partial_3\cA u)(\bx,z) &= -\frac{2}{z} \int_{\Tref} \rho(\bxi) [c - u(F_{(\bx,z)}(\bxi))]\,d\bxi
    - \frac{1}{z} \int_{\Tref} \nabla\rho(\bxi) \cdot \bxi \cdot [c - u(F_{(\bx,z)}(\bxi))]\,d\bxi,
  \end{align*}
  and inductively we can conclude for $|\bk|=k\in\N$, $k\geq 1$, the basic estimates
  \begin{align}
      | \cA u (\bx,z) | 
      &\lesssim \frac{1}{z^{2}} \int_{\bx+\frac{z}{2} \Tref} |u(\bs)|\,d\bs,\label{eq:lemma:lifting-from-triangle:1}\\
      | \partial^{\bk} \cA u (\bx,z) | 
      &\leq C_k \frac{1}{z^{2+k}} \min_{c\in\R}\int_{\bx+\frac{z}{2} \Tref} |c - u(\bs)|\,d\bs.\label{eq:lemma:lifting-from-triangle:2}
  \end{align}

  \textbf{Proof of~(\ref{item:lemma:lifting-from-triangle-1}).}
  This follows at once as $(\bx,z)\mapsto F_{(\bx,z)}(\bxi)$ is affine for fixed $\bxi$.

  \textbf{Proof of~(\ref{item:lemma:lifting-from-triangle-2}).}
  This follows by inspection.

  \textbf{Proof of~(\ref{item:lemma:lifting-from-triangle-3}).}
  The estimate~\eqref{eq:lemma:lifting-from-triangle:1} and Cauchy-Schwarz imply
  \begin{align*}
    |\cA  u(\bx,z)|^2 \lesssim \frac{1}{z^2} \int_{\frac{z}{2} \Tref} \abs{u\left( \bx+\bs \right)}^2\,d\bs
    \leq \frac{1}{z^2} \int_{z \Tref} \abs{u\left( \bx+\bs \right)}^2\,d\bs.
  \end{align*}
  Using Fubini, we get
  \begin{align*}
    \| d_{\Tref\times\{0\}}^\gamma\cA u\|_{L^2(\tetref)}^2 
    &\lesssim \int_{z=0}^1 \frac{1}{z^{2-2\gamma}} \int_{z\Tref} \int_{(1-z)\Tref}\abs{u\left( \bx+\bs \right)}^2\,d\bx\,d\bs\,dz.
  \end{align*}
  As $\Tref$ is convex, $\bx\in (1-z)\Tref$ and $\bs\in z\Tref$ imply $\bx+\bs \in\Tref$, and we conclude
  \begin{align*}
    \| d_{\Tref\times\{0\}}^\gamma\cA u\|_{L^2(\tetref)}^2 \lesssim \|u\|_{L^2(\Tref)}^2 \int_{z=0}^1
    \frac{1}{z^{2-2\gamma}} \int_{z \Tref}\,d\bs\,dz
    \leq C_\gamma \|u\|_{L^2(\Tref)}^2.
  \end{align*}

  \textbf{Proof of~(\ref{item:lemma:lifting-from-triangle-4}).}
  Let $T' = \left\{ \bx\in\Tref \mid \exists z \text{ such that } (\bx,z)\in\widehat f \right\}$.
  Consider $z$ as a function of $\bx$.
  Note that $\bx\mapsto z(\bx)$ is affine and has the form
  $z(u,v) = 1 + u z_x + v z_y$ for some $z_x,z_y\in\R$.
  For $\bx\in T'$, it holds $d_{\widehat e}(\bx) \sim z(\bx)$ and hence also
  \begin{align*}
    d_{\widehat e}(\bs) \sim z(\bx) \quad\text{ for all } \bs\in
    \bx + \frac{z(\bx)}{2}\Tref.
  \end{align*}
  We conclude
  \begin{align*}
    \| d_{\widehat e}^\gamma \cA  u \|_{L^2(\widehat f)}^2 &\sim \int_{T'} z(\bx)^{2\gamma} | \cA  u (\bx,z(\bx))|^2\,d\bx
    \lesssim \int_{T'} z(\bx)^{2\gamma-2} \int_{\bx+\frac{z(\bx)}{2} \Tref} \abs{u(\bs)}^2\,d\bs\,d\bx\\
    &\lesssim \int_{T'} z(\bx)^{-2} \int_{\bx+\frac{z(\bx)}{2} \Tref} d_{\widehat e}(\bs)^{2\gamma}\abs{u(\bs)}^2\,d\bs\,d\bx\\
    &\lesssim \int_{T'} \int_{\Tref} f(\bx + z(\bx)\bxi/2)\,d\bxi\,d\bx,
  \end{align*}
  where we wrote $f(\bs) := d_{\widehat e}(\bs)^{2\gamma}\abs{u(\bs)}^2$ and used the substitution $\bs = \bx+z(\bx)\bxi/2$.
  We apply Fubini and the substitution $\bx' = \bx+z(\bx)\bxi/2$. Note that for $\bx\in T'$ and $\bxi\in \Tref$ it holds $\bx'\in \Tref$,
  so that
  \begin{align*}
    \| d_{\widehat e}^\gamma \cA  u \|_{L^2(\widehat f)}^2 &\lesssim \int_{\Tref} \int_{T'} f(\bx + z(\bx)\bxi/2)\,d\bx\,d\bxi\\
    &\lesssim \int_{\Tref} \int_{\Tref} \frac{f(\bx')}{\abs{1+\xi_1 z_x/2 + \xi_2 z_y/2}}\,d\bx'\,d\xi_1 d\xi_2\\
    &\lesssim \| d_{\widehat e}^\gamma u \|_{L^2(\Tref)}^2.
  \end{align*}
  The last estimate follows from the fact that $\abs{1+\xi_1 z_x/2 + \xi_2 z_y/2} = z(\bxi/2)$, which is bounded from below away from zero uniformly
  in $\bxi\in\Tref$.

  \textbf{Proof of~(\ref{item:lemma:lifting-from-triangle-5}).}
  Parametrize $\widehat e$ by $z\in(0,1)\mapsto (\bx(z),z)=(\bx,z)$ and note $d_{\widehat V}(\bx,z)\sim z$ on $\widehat e$.
  Furthermore, there is a constant $\alpha>0$ such that, extending $u$ by zero outside $\Tref$,
  \begin{align*}
    |\cA  u(\bx,z)| \lesssim 
    \frac{1}{z^2} \int_{B_{z}(0)} | u(\widehat V + \alpha\bs) |\,d\bs = 
    \frac{1}{z^2} \int_0^z\int_0^{2\pi} | u(\widehat V + \alpha (r\cos\phi,r\sin\phi))|r\,d\phi\,dr.
  \end{align*}
  The weighted Hardy inequality~\cite[Thm.~I.9.16]{Zygmund} for $2\gamma -2<1$ and H\"older show
  \begin{align*}
    \| d_{\widehat V}^\gamma \cA u \|_{L^2(\widehat e)}^2
    &\lesssim\int_0^1 z^{2\gamma-2} \left( z^{-1} \int_0^z\int_0^{2\pi} | u(\widehat V + \alpha (r\cos\phi,r\sin\phi))|r\,d\phi\,dr \right)^2\,dz\\
    &\lesssim\int_0^1 z^{2\gamma} \int_0^{2\pi} | u(\widehat V + \alpha (z\cos\phi,z\sin\phi))|^2\,d\phi\,dz\\
    &\lesssim \| d_{\widehat V}^{\gamma - 1/2} u \|_{L^2(\Tref)}^2.
  \end{align*}
  \textbf{Proof of~(\ref{item:lemma:lifting-from-triangle-6}).} We follow~\cite[Lemma 1.4.1.4]{Grisvard} and
  note that by~\eqref{eq:lemma:lifting-from-triangle:2} and Cauchy-Schwarz
  \begin{align*}
    \| d_{\Tref\times\{0\}}^{k-1/2-\theta} \partial^{\bk} \cA u \|_{L^2(\tetref)}^2
    &\lesssim \int_{\tetref} z^{-3-2\theta} \int_{\bx+\frac{z}{2}\Tref} |u(\bx) - u(\bs)|^2\,d\bs\,d(\bx,z)\\
    &\lesssim \int_{\Tref} \int_{\Tref} |u(\bx) - u(\bs)|^2 \int_{|\bx-\bs|/2}^\infty z^{-3-2\theta}\,dz\,d\bx\,d\bs
    \lesssim | u |_{H^\theta(\Tref)}^2.
  \end{align*}
  To treat the second term on the left-hand side of~(\ref{item:lemma:lifting-from-triangle-6}),
  we use~\eqref{eq:lemma:lifting-from-triangle:1} and the notation introduced
  in the proof of~(\ref{item:lemma:lifting-from-triangle-4}). With Cauchy-Schwarz we calculate
  \begin{align*}
    \| d_{\Tref\times\{0\}}^{k-\theta} \partial^{\bk}\cA u \|_{L^2(\widehat f)}^2 &\sim
    \int_{T'} z(\bx)^{-2-2\theta} \int_{\bx+\frac{z(\bx)}{2} \Tref} |u(\bx) - u(\bs)|^2\,d\bs\,d\bx\\
    &\lesssim \int_{T'} \int_{\bx+\frac{z(\bx)}{2} \Tref} \frac{|u(\bx) - u(\bs)|^2}{|\bx-\bs|^{2+2\theta}}\,d\bs\,d\bx
    \lesssim | u |_{H^\theta(\Tref)}^2.
  \end{align*}
  To treat the third term on the left-hand side of~(\ref{item:lemma:lifting-from-triangle-6}),
  suppose that $\widehat e$ is edge of the lateral face $\widehat f$.
  From the one-dimensional trace inequality
  \begin{align*}
    x|v(x,0)|^2 \lesssim \int_0^x |v(x,y)|^2\,dy + x^2 \int_0^x |\partial_y v(x,s)|^2\,ds
  \end{align*}
  we conclude the trace inequality
  \begin{align*}
    \| d_{\widehat V}^{k+1/2-\theta} \partial^{\bk} \cA u \|_{L^2(\widehat e)}^2 \lesssim 
    \| d_{\Tref\times\{0\}}^{k-\theta} \partial^{\bk} \cA u \|_{L^2(\widehat f)}^2 +
    \sum_{|\bk'|=k+1}\| d_{\Tref\times\{0\}}^{(k+1)-\theta} \partial^{\bk'} \cA u \|_{L^2(\widehat f)}^2,
  \end{align*}
  and the result follows using the estimate for the second term on the left.
%
%

  \textbf{Proof of~(\ref{item:lemma:lifting-from-triangle-7}).} For $j=0$ this follows
    immediately from formula~\eqref{eq:lemma:lifting-from-triangle:1}, and for $j\geq 1$ from formula~\eqref{eq:lemma:lifting-from-triangle:2}.

  \textbf{Proof of~(\ref{item:lemma:lifting-from-triangle-8}).}
  The formulas~\eqref{eq:lemma:lifting-from-triangle:1} and~\eqref{eq:lemma:lifting-from-triangle:2}
  and H\"older's inequality show that
  \begin{align}\label{eq:lemma:lifting-from-triangle:3}
    | \partial^\bk \cA  u (\bx,z)| \lesssim z^{-1-\abs{\bk}} \| u \|_{L^2(\bx+\frac{z}2\Tref)}.
  \end{align}
  Using the Bramble-Hilbert Lemma~\cite[Lem.~4.3.8]{BrennerScott}, and the H\"older inequality we
  can even conclude from formula~\eqref{eq:lemma:lifting-from-triangle:2}
  \begin{align}\label{eq:lemma:lifting-from-triangle:4}
    | \nabla \cA  u (\bx,z)| &\lesssim z^{-3} \min_{c\in\R} \int_{\bx+\frac{z}{2}\Tref} | c - u(\bs) |\,d\bs
    \lesssim z^{-2} \int_{\bx+\frac{z}{2}\Tref} | \nabla u(\bs) |\,d\bs
    \lesssim z^{-1} \| \nabla u \|_{L^2(\bx+\frac{z}2\Tref)}.
  \end{align}
  For $z$ sufficiently small (depending on $\delta-\varepsilon$) and $\bx\in (1-z)\Tref_\varepsilon$
  we have $F_{(\bx,z)}(\Tref) = \bx+\frac{z}2\Tref \subset \Tref_\delta$, and we conclude with the aforegoing estimates
  (\eqref{eq:lemma:lifting-from-triangle:3} for $|\bk|=0$ and~\eqref{eq:lemma:lifting-from-triangle:4} for $|\bk|=1$)
  \begin{align}\label{eq:lemma:lifting-from-triangle:5}
    \| \partial^\bk \cA  u(\cdot,z) \|_{L^\infty( (1-z)\Tref_\varepsilon)} \lesssim \| u \|_{W^{\abs{\bk},\infty}(\Tref_\delta)}.
  \end{align}
  The combination of~\eqref{eq:lemma:lifting-from-triangle:5} for $z$ sufficiently small and~\eqref{eq:lemma:lifting-from-triangle:3}
  for $z$ not sufficiently small proves the desired estimates.
\end{proof}
The operator $\cA$ can be modified to vanish on lateral faces of $\tetref$ if $u$ vanishes on the corresponding
bottom edges of $\Tref$.
\begin{lemma}\label{lemma:lifting-from-triangle:bc}
  Let $\widehat\cE \subset \{ \widehat e_{4}, \widehat e_5, \widehat e_6 \}$ and let
  $\widehat\cF = \{ \widehat f_{j-3}\mid \widehat e_j\in \widehat\cE \}$ be all lateral faces with edge in $\widehat\cE$.
  There exists a linear operator $\cA _{\widehat \cE}: L^1_{loc}(\Tref)\rightarrow C^{\infty}(\tetref)$ with the following properties:
  \begin{enumerate}[(i)]
    \item \label{item:lemma:lifting-from-triangle:bc-1} If $u$ is continuous at a point $\bx\in\Tref$, then $(\cA _{\what\cE} u)(\bx,0)=u(\bx)$.
    \item \label{item:lemma:lifting-from-triangle:bc-2} 
      For every $\gamma > -1/2$ there is a constant $C_\gamma$ such that
      \begin{align*}
	\| d_{\Tref\times\{0\} }^\gamma \cA _{\what\cE}u \|_{L^2(\tetref)} \leq C_\gamma \| u \|_{L^2(\Tref)}.
      \end{align*}
    \item \label{item:lemma:lifting-from-triangle:bc-3}
      The function $\cA _{\what\cE} u$ vanishes on all faces in $\widehat \cF$.
    \item \label{item:lemma:lifting-from-triangle:bc-4} If $u$ is a polynomial of degree $p\geq \#\widehat\cE$
      that vanishes on all edges in $\widehat\cE$,
      then $\cA _{\widehat\cE} u$ is a polyonomial of degree $p$.
    \item \label{item:lemma:lifting-from-triangle:bc-5}
      For every $s\in(0,1)$ there is a constant $C_s>0$ such that
      \begin{align*}
	\| d_{\Tref\times \{0\}}^{1/2-s} \nabla \cA _{\widehat\cE}u \|_{L^2(\tetref)}
	\leq C_s \left( |u|_{H^s(\Tref)} + \| d_{\widehat \cE}^{-s}u \|_{L^2(\Tref)}  \right).
      \end{align*}
    \item \label{item:lemma:lifting-from-triangle:bc-6} For every $\varepsilon>0$ and $j\in\N\cup\left\{ 0 \right\}$, there is $C_{\varepsilon,j}>0$
      such that
      \begin{align*}
	\| \cA_{\what\cE}  u \|_{W^{j,\infty}(\tetref_\varepsilon)} \leq C_{\varepsilon,j} \| u \|_{L^2(\Tref)}.
      \end{align*}
    \item \label{item:lemma:lifting-from-triangle:bc-7}
      For $\varepsilon>0$ consider the set $\Tref_\varepsilon$
      and for $z\in[0,1]$ the scaled versions $(1-z)\Tref_\varepsilon$.
      Then, for sufficiently small $\varepsilon>0$ there is a $\delta>\varepsilon$ depending only on $\varepsilon$ such that
      for $\bk\in\N_0^3$ with $\abs{\bk}\leq 1$ there holds for any $u$ vanishing on $\widehat\cE$ that
      \begin{align*}
	\| \partial^\bk \cA_{\widehat\cE} u (\cdot,z) \|_{L^{\infty}( (1-z)\Tref_\varepsilon)}
	\lesssim \| u \|_{L^2(\Tref)} + \| u \|_{W^{\abs{\bk},\infty}(\Tref_\delta)},
      \end{align*}
    with a constant depending only on $\varepsilon$.
  \end{enumerate}
\end{lemma}
\begin{proof}
  If $\widehat\cE=\emptyset$, then we set $\cA _{\widehat\cE}=\cA $ the operator from Lemma~\ref{lemma:lifting-from-triangle}.
  If $\widehat\cE$ is not empty, the construction will be carried out in several steps.

  \textbf{Step 1.} Define
  \begin{align*}
    u_1(x,y,z) = \cA  u(x,y,z) - z \cA u(\widehat \bv_4),
  \end{align*}
  note that $u_1(\widehat \bv_4)=0$ and therefore also
  \begin{align}\label{eq:lemma:lifting-from-triangle:bc:17}
    f_k\in\widehat\cF \implies u_1\circ\Pi_{\widehat e_k}|_{\widehat f_k}=u_1(\widehat \bv_4)=0.
  \end{align}
  \textbf{Step 2.} We will subtract edge contributions corresponding to all lateral edges of all faces in $\widehat \cF$.
  To that end, define the corresponding indices
  $N = \{j \in \{1,2,3\} \mid \widehat e_j \text{ is lateral edge of } \widehat f\in \widehat \cF\}$.
  For $j\in N$, let $p_{j}(x,y,z)=0$ be the affine equation of the hyperplane orthogonal to $\widehat e_j$ and passing
  through the point $\widehat \bv_j$, and for convenience let $p_j$ be positive on $\tetref$.
  We claim that there is a constant $c_j\neq 0$ such that for $(x,y,z)\in \widehat e_j$ it holds
  $p_j(x,y,z)= c_j z$.
  Indeed, write $p_j(x,y,z) = \ell_j(x,y,z) + d_j$ with $\ell_j$ linear and $d_j\in\R$. Parametrize $\widehat e_j$ by $z\mapsto \widehat \bv_j + z \bn_j$
  and calculate $p_j(\widehat \bv_j+z\bn_j) = \ell_j(\widehat \bv_j) + z\ell_j(\bn_j) + d_j = z\ell_j(\bn_j)$.
  Define
  \begin{align*}
    u_2(x,y,z) = u_1(x,y,z) - \sum_{j\in N}
    c_j \frac{z}{p_j(x,y,z)} u_1(\Pi_{\widehat e_j}(x,y,z)),
  \end{align*}
  where $\Pi_{\widehat e_j}$ is the affine function calculating the orthogonal projection onto the line spanned by $\widehat e_j$.
  Note that $p_j(x,y,z)$ is proportional to the distance of $(x,y,z)\in\tetref$ to the hyperplane $p_j=0$, and
  hence
  \begin{align}\label{eq:lemma:lifting-from-triangle:bc:10}
    p_j(x,y,z) \sim d_{\widehat\bv_j}(x,y,z) > z \qquad\text{ for } (x,y,z)\in\tetref.
  \end{align}
  Clearly, $|\nabla p_j(x,y,z)|\sim 1$, and this shows that
  \begin{align}\label{eq:lemma:lifting-from-triangle:bc:11}
    |\nabla \frac{z}{p_j(x,y,z)}| \lesssim \frac{1}{d_{\widehat\bv_j}(x,y,z)} \qquad\text{ for } (x,y,z)\in\tetref.
  \end{align}

  \textbf{Step 3.} We will subtract face contributions corresponding to all faces in $\widehat \cF$.
  Let $\widehat f\in\widehat \cF$ be contained in the plane given by the affine equation $p_{\widehat f}(x,y)-z=0$,
  and for convenience let $p_{\widehat f}(x,y)-z$ be positive on $\tetref$.
  Let $\Pi_{\widehat f}$ be the affine function calculating
  the orthogonal projection onto this plane. Define
  \begin{align*}
    \cA _{\widehat \cE} u(x,y,z) = u_2(x,y,z) - \sum_{\widehat f\in \widehat \cF} \frac{z}{p_{\widehat f}(x,y)} u_2(\Pi_{\widehat f}(x,y,z)).
  \end{align*}
  For $\widehat f_k\in\what\cF$, as $p_{\widehat f_k}(x,y)-z$ is positive on $\tetref$, it holds
  \begin{align}\label{eq:lemma:lifting-from-triangle:bc:18}
    \frac{z}{p_{\widehat f_k}(x,y)}<1.
  \end{align}
  For a point $(x,y,z)\in\tetref$, consider angles $\sin\alpha= z / d_{\what e_{k+3}}(x,y,z)$ and
  $\sin\beta = d_{\what f_k}(x,y,z)/d_{\what e_{k+3}}(x,y,z)$.
  Recall that $\widehat e_{k+3}$ is the edge that $\widehat f_k$ shares with the base $\Tref\times\left\{ 0 \right\}$
  of the tetrahedron, and hence $\alpha+\beta$ is constant.
  Furthermore, $p_{\widehat f_k}(x,y)-z$ is proportional to the distance of $(x,y,z)\in\tetref$ to the plane spanned by $\widehat f_k$, and hence
  \begin{align}\label{eq:lemma:lifting-from-triangle:bc:12}
    p_{\widehat f_k}(x,y) \geq C d_{\widehat f_k}(x,y,z) + z = d_{\widehat e_{k+3}}(x,y,z)\left( C\sin\beta+\sin\alpha \right)\gtrsim d_{\widehat e_{k+3}}(x,y,z).
  \end{align}
  Now, with~\eqref{eq:lemma:lifting-from-triangle:bc:18},
  \begin{align*}
    |\partial_x \frac{z}{p_{\widehat f_k}(x,y)}|
    +|\partial_y \frac{z}{p_{\widehat f_k}(x,y)}|
    = |\frac{z\partial_x p_{\what f_k}(x,y)}{p_{\what f_k}(x,y)^2}| 
    + |\frac{z\partial_y p_{\what f_k}(x,y)}{p_{\what f_k}(x,y)^2}| 
    \lesssim \frac{1}{p_{\what f_k}(x,y)},
  \end{align*}
  and also
  \begin{align*}
    |\partial_z \frac{z}{p_{\widehat f_k}(x,y)}| = \frac{1}{p_{\what f_k}(x,y)}.
  \end{align*}
  We conclude, using~\eqref{eq:lemma:lifting-from-triangle:bc:12},
  \begin{align}\label{eq:lemma:lifting-from-triangle:bc:13}
    |\nabla \frac{z}{p_{\widehat f_k}(x,y)}| \lesssim \frac{1}{d_{\widehat e_{k+3}}(x,y,z)}.
  \end{align}

  The operator $\cA _{\widehat \cE}$ is clearly linear and $\cA _{\widehat\cE} u\in C^{\infty}(\tetref)$.

  \textbf{Proof of~(\ref{item:lemma:lifting-from-triangle:bc-1}).} This follows by construction, as $\cA u(\bx)=u(\bx)$ due to
  Lemma~\ref{lemma:lifting-from-triangle} (\ref{item:lemma:lifting-from-triangle-1}), and the correction terms are all multiplied by $z$.

  \textbf{Proof of~(\ref{item:lemma:lifting-from-triangle:bc-2}).} Due to
  Lemma~\ref{lemma:lifting-from-triangle} (\ref{item:lemma:lifting-from-triangle-3}), (\ref{item:lemma:lifting-from-triangle-4}),
  (\ref{item:lemma:lifting-from-triangle-5}) and (\ref{item:lemma:lifting-from-triangle-7}), it follows for $-1/2<\gamma$
  \begin{align}\label{eq:lemma:lifting-from-triangle:bc:3}
  \| d_{\Tref\times\{0\}}^\gamma u_1 \|_{L^2(\tetref)} + \| d_{\widehat e_{3+k}}^{1/2+\gamma} u_1 \|_{L^2(\widehat f_k)}
  + \| d_{\widehat \bv_j}^{1+\gamma} u_1 \|_{L^2(\widehat e_j)}
    \lesssim \| u \|_{L^2(\Tref)}
  \end{align}
  for $j\in N$ and $\widehat f_k\in\widehat \cF$. Then,
  \begin{align}\label{eq:lemma:lifting-from-triangle:bc:4}
    \begin{split}
    \| d_{\Tref\times\{0\}}^\gamma u_2 \|_{L^2(\tetref)}
    \lesssim \| d_{\Tref\times\{0\}}^\gamma u_1 \|_{L^2(\tetref)} + \sum_{j\in N} \| d_{\Tref\times\{0\}}^\gamma u_1\circ \Pi_{\widehat e_j} \|_{L^2(\tetref)}
    &\lesssim \| u \|_{L^2(\Tref)},
    \end{split}
  \end{align}
  where we used~\eqref{eq:lemma:lifting-from-triangle:bc:10} in the first inequality,
  and~\eqref{eq:lemma:lifting-from-triangle:bc:1} (with $\alpha=0,\beta=\gamma$) and~\eqref{eq:lemma:lifting-from-triangle:bc:3} in the second one.
  On $\what f_k$ there holds $d_{\what e_{3+k}}\sim d_{\Tref}$, so that we obtain
  \begin{align}\label{eq:lemma:lifting-from-triangle:bc:15}
    \begin{split}
    \| d^{1/2+\gamma}_{\widehat e_{3+k}} u_2 \|_{L^2(\widehat f_k)}
    &\lesssim \| d^{1/2+\gamma}_{\widehat e_{3+k}} u_1 \|_{L^2(\widehat f_k)}
    + \sum_{\substack{j\in N\\j\neq k}} \| d^{1/2+\gamma}_{\widehat e_{3+k}} u_1\circ \Pi_{\widehat e_j} \|_{L^2(\widehat f_k)}\\
    &\lesssim \| d^{1/2+\gamma}_{\widehat e_{3+k}} u_1 \|_{L^2(\widehat f_k)}
    + \sum_{\substack{j\in N\\j\neq k}} \| d^{1/2+\gamma}_{\Tref} u_1\circ \Pi_{\widehat e_j} \|_{L^2(\widehat f_k)}
    \lesssim \| u \|_{L^2(\Tref)},
    \end{split}
  \end{align}
  where we used~\eqref{eq:lemma:lifting-from-triangle:bc:17} and~\eqref{eq:lemma:lifting-from-triangle:bc:10} in the first inequality,
  and~\eqref{eq:lemma:lifting-from-triangle:bc:1} (with $\alpha=0,\beta=\gamma$) and~\eqref{eq:lemma:lifting-from-triangle:bc:3} in the
  last one.
  We conclude with~\eqref{eq:lemma:lifting-from-triangle:bc:18},~\eqref{eq:lemma:lifting-from-triangle:bc:2}
  and~\eqref{eq:lemma:lifting-from-triangle:bc:4},~\eqref{eq:lemma:lifting-from-triangle:bc:15}
  \begin{align*}
    \| d_{\Tref\times\{0\}}^\gamma \cA _{\what\cE}u \|_{L^2(\tetref)}
    &\lesssim \| d_{\Tref\times\{0\}}^\gamma u_2 \|_{L^2(\tetref)} + \sum_{\widehat f_k\in\widehat\cF}
    \| d_{\Tref\times\{0\}}^\gamma u_2\circ \Pi_{\widehat f_k} \|_{L^2(\tetref)}\\
    &\lesssim \| d_{\Tref\times\{0\}}^\gamma u_2 \|_{L^2(\tetref)} + \sum_{\widehat f_k\in\widehat\cF}
    \| d_{\widehat e_{3+k}}^{1/2+\gamma} u_2 \|_{L^2(\widehat f_k)}
    \lesssim \| u \|_{L^2(\Tref)}.
  \end{align*}

  \textbf{Proof of~(\ref{item:lemma:lifting-from-triangle:bc-3}).} According to Step 1, $u_1$ vanishes in $\widehat\bv_4$.
  We will now show that $u_2$ vanishes on all edges with indices in $N$. To that end, let $j\in N$. For $(x,y,z)\in \widehat e_j$ it holds
  according to Step 2 that $c_j z = p_j(x,y,z)$, as well as $u_1(\Pi_{\widehat e_k}(x,y,z)) = u_1(\widehat \bv_4)=0$ for
  $k\neq j\in N$. Hence, for $(x,y,z)\in \widehat e_j$,
  \begin{align}\label{eq:lemma:lifting-from-triangle:bc:22}
    \begin{split}
    u_2(x,y,z) &= u_1(x,y,z) - c_j \frac{z}{p_j(x,y)} u_1\circ\Pi_{\widehat e_j}(x,y,z) +
    \sum_{\substack{k\in N\\k\neq j}} c_k \frac{z}{p_k(x,y)} u_1\circ\Pi_{\widehat e_k}(x,y,z)\\
    &= u_1(x,y,z) - u_1(x,y,z) -
    \sum_{\substack{k\in N\\k\neq j}} c_k \frac{z}{p_k(x,y)} u_1(\widehat \bv_4) = 0.
    \end{split}
  \end{align}
  Next, we will show that $\cA _{\widehat \cE} u$ vanishes on all faces in $\widehat\cF$.
  To that end, let $\widehat f_j\in\widehat\cF$. For $(x,y,z)\in \widehat f_j$ it holds
  $z = p_{\widehat f_j}(x,y)$.
  Furthermore, if $\widehat f_k\in\widehat \cF$ with $k\neq j$, then $\ell_{j,k}\in N$ for the lateral edge $\widehat e_{\ell_{j,k}}$ which is shared
  by $\what f_k$ and $\what f_j$.
  Hence, for $(x,y,z)\in\what f_j$,~\eqref{eq:ort2} implies
  $u_2(\Pi_{\widehat f_k}(x,y,z)) = u_2(\Pi_{\widehat e_{\ell_{j,k}}}(x,y,z)) = 0$,
  as we have already demonstrated that $u_2$ vanishes on all edges with indices in $N$. Hence, for $(x,y,z)\in \widehat f_j \in \widehat\cF$,
  \begin{align*}
    \cA_{\widehat \cE} u (x,y,z) &= u_2(x,y,z) - \frac{z}{p_{\widehat f_j}(x,y)} u_2(\Pi_{\widehat f_j}(x,y,z)) -
    \sum_{\substack{\widehat f_k\in \widehat \cF\\k\neq j}} \frac{z}{p_{\widehat f_k}(x,y)} u_2(\Pi_{\widehat f_k}(x,y,z))\\
    &= u_2(x,y,z) - u_2(x,y,z) -
    \sum_{\substack{\widehat f_k\in \widehat \cF\\k\neq j}} \frac{z}{p_{\widehat f_k}(x,y)} u_2(\Pi_{\widehat e_{\ell_{j,k}}}(x,y,z)) = 0.
  \end{align*}

  \textbf{Proof of~(\ref{item:lemma:lifting-from-triangle:bc-4}).} If $u$ is polynomial of degree $p$, then so is $\cA u$, and hence also
  $u_1$. Furthermore, $\cA u$, and hence $u_1$, vanish on $\widehat \cE$ if $u$ does.

  As $\Pi_{\widehat e_j}$ is affine, $p_j(x,y,z)=0$ implies $u_1(\Pi_{\widehat e_j}(x,y,z))=u_1(\widehat \bv_j) = 0$.
  Polynomial division shows that $u_2$ is indeed a polynomial of degree $p$, and due to construction it vanishes on $\widehat \cE$,
  cf.~\eqref{eq:lemma:lifting-from-triangle:bc:22}.

  Finally, note that $\Pi_{\widehat f_j}$ is affine, and $p_{\widehat f_j}(x,y)=0$ implies $(x,y,0)\in\widehat f_j$ and hence
  $u_2(\Pi_{\widehat f_j}(x,y,0)) = u_2(x,y,0)=0$. Polynomial division shows that $A_{\what \cE}u$ is indeed a polynomial of degree $p$.

  \textbf{Proof of~(\ref{item:lemma:lifting-from-triangle:bc-5}).} At first, various applications of Lemma~\ref{lemma:lifting-from-triangle} give
  for $j\in N$ and $\widehat f_k\in\widehat\cF$
  \begin{align}\label{eq:lemma:lifting-from-triangle:bc:5}
    \begin{split}
    \| d_{\Tref\times \{0\}}^{1/2-s} \nabla u_1 \|_{L^2(\tetref)}
    + \| d_{\Tref\times \{0\}}^{1/2-s} u_1 \|_{L^2(\widehat e_j)}
    + \| d_{\Tref\times \{0\}}^{3/2-s} \nabla u_1 \|_{L^2(\widehat e_j)}
    &\lesssim | u |_{H^s(\Tref)} +\| d_{\widehat\cE}^{-s}u \|_{L^2(\Tref)}\\
    \| d_{\Tref\times\{0\}}^{-s} u_1 \|_{L^2(\widehat f_k)} + \| d_{\Tref\times \{0\}}^{1-s} \nabla u_1 \|_{L^2(\widehat f_k)}
    &\lesssim | u |_{H^s(\Tref)} + \| d_{\widehat \cE}^{-s}u \|_{L^2(\Tref)}.
    \end{split}
  \end{align}
  Specifically, the terms without derivatives on $u_1$ are bounded by
  Lemma~\ref{lemma:lifting-from-triangle},~(\ref{item:lemma:lifting-from-triangle-4}) and~(\ref{item:lemma:lifting-from-triangle-5}),
  while the terms containing $\nabla u_1$ are bounded by Lemma~\ref{lemma:lifting-from-triangle},~(\ref{item:lemma:lifting-from-triangle-6}).
  Note that $\Pi_{\widehat e_j}$ is affine, so that
  \begin{align}\label{eq:lemma:lifting-from-triangle:bc:9}
    | \nabla (u_1\circ \Pi_{\widehat e_j})(x,y,z)| \sim | (\nabla u_1)\circ \Pi_{\widehat e_j}(x,y,z)|.
  \end{align}
  We conclude with~\eqref{eq:lemma:lifting-from-triangle:bc:1}
  (with $\alpha=-1,\beta=1/2-s$ as well as $\alpha=0,\beta=1/2-s$),~\eqref{eq:lemma:lifting-from-triangle:bc:5},
  and~\eqref{eq:lemma:lifting-from-triangle:bc:10},~\eqref{eq:lemma:lifting-from-triangle:bc:11}
  \begin{align}\label{eq:lemma:lifting-from-triangle:bc:6}
    \begin{split}
    \| d_{\Tref\times \{0\}}^{1/2-s} \nabla u_2 \|_{L^2(\tetref)}
    &\leq \| d_{\Tref\times \{0\}}^{1/2-s} \nabla u_1 \|_{L^2(\tetref)}\\
    &\qquad+ \sum_{j\in N}
    \left( 
      \| d_{\widehat\bv_j}^{-1}d_{\Tref\times \{0\}}^{1/2-s} u_1\circ\Pi_{\widehat e_j} \|_{L^2(\tetref)}
      + \| d_{\Tref\times \{0\}}^{1/2-s} \left( \nabla u_1 \right)\circ\Pi_{\widehat e_j} \|_{L^2(\tetref)}
    \right)\\
    &\leq \| d_{\Tref\times \{0\}}^{1/2-s} \nabla u_1 \|_{L^2(\tetref)}
    + \sum_{j\in N}
    \left( 
    \| d_{\Tref\times \{0\}}^{1/2-s} u_1 \|_{L^2(\widehat e_j)}
    + \| d_{\Tref\times \{0\}}^{3/2-s} \nabla u_1 \|_{L^2(\widehat e_j)}
    \right)\\
    &\stackrel{\eqref{eq:lemma:lifting-from-triangle:bc:5}}\lesssim |u|_{H^s(\Tref)} + \| d_{\widehat \cE}^{-s}u \|_{L^2(\Tref)}.
    \end{split}
  \end{align}
  Furthermore, for $(x,y,z)\in\tetref$ we have, cf.~\eqref{eq:lemma:lifting-from-triangle:bc:10},
    \begin{align*}
      d_{\Tref}^{-s} \frac{z}{p_j(x,y,z)} \sim d_{\Tref}^{1-s}(x,y,z)d_{\what\bv_j}^{-1}(x,y,z).
    \end{align*}
    Together with~\eqref{eq:lemma:lifting-from-triangle:bc:17},~\eqref{eq:lemma:lifting-from-triangle:bc:1}
    (with $\alpha=-1,\beta=1/2-s$), and~\eqref{eq:lemma:lifting-from-triangle:bc:5} we then conclude
  \begin{align}\label{eq:lemma:lifting-from-triangle:bc:7}
    \| d_{\Tref\times \{0\}}^{-s} u_2 \|_{L^2(\widehat f_k)}
    \leq \| d_{\Tref\times \{0\}}^{-s} u_1 \|_{L^2(\widehat f_k)}
    + \sum_{\substack{j\in N\\j\neq k}} \| d_{\Tref\times \{0\}}^{1/2-s} u_1 \|_{L^2(\widehat e_j)}
    \lesssim |u|_{H^s(\Tref)} + \| d_{\widehat\cE}^{-s}u \|_{L^2(\Tref)}.
  \end{align}
  Note that, according to Lemma~\ref{lemma:lifting-from-triangle},~\eqref{item:lemma:lifting-from-triangle-7},
  \begin{align*}
    | \nabla u_1(\widehat\bv_4) | \leq | \nabla \cA u(\widehat\bv_4) | + | \cA u(\widehat\bv_4) | \lesssim \| u \|_{L^2(\Tref)}.
  \end{align*}
  As $d_{\what\bv_k}^{-1} \lesssim 1$ on $\what f_k$, we conclude
  with~\eqref{eq:lemma:lifting-from-triangle:bc:11},~\eqref{eq:lemma:lifting-from-triangle:bc:9},~\eqref{eq:lemma:lifting-from-triangle:bc:1} (with $\alpha=-1,\beta=1/2-s$ as well as $\alpha=0,\beta=1/2-s$),~\eqref{eq:lemma:lifting-from-triangle:bc:16} (with $\beta=1-s$), and~\eqref{eq:lemma:lifting-from-triangle:bc:5}
  \begin{align}\label{eq:lemma:lifting-from-triangle:bc:8}
    \begin{split}
    \| d_{\Tref\times \{0\}}^{1-s} \nabla u_2 \|_{L^2(\widehat f_k)}
    &\lesssim
    \| d_{\Tref\times \{0\}}^{1-s} \nabla u_1 \|_{L^2(\widehat f_k)}\\
    &\qquad+ \sum_{\substack{j\in N\\j\neq k}}
    \left( 
      \| d_{\widehat \bv_j}^{-1} d_{\Tref\times \{0\}}^{1-s}u_1\circ\Pi_{\widehat e_j} \|_{L^2(\widehat f_k)}
      +\| d_{\Tref\times \{0\}}^{1-s} \left( \nabla u_1 \right)\circ \Pi_{\widehat e_j} \|_{L^2(\widehat f_k)}
    \right)\\
    &\qquad
    + \| d_{\widehat \bv_{k}}^{-1} d_{\Tref\times \{0\}}^{1-s}u_1\circ\Pi_{\widehat e_k} \|_{L^2(\widehat f_k)}
    + \| d_{\Tref\times \{0\}}^{1-s} \left( \nabla u_1 \right)\circ \Pi_{\widehat e_k} \|_{L^2(\widehat f_k)}
    \\
    &\lesssim
    \| d_{\Tref\times \{0\}}^{1-s} \nabla u_1 \|_{L^2(\widehat f_k)}\\
    &\qquad+ \sum_{\substack{j\in N\\j\neq k}}
    \left( 
      \| d_{\Tref\times \{0\}}^{1/2-s}u_1 \|_{L^2(\widehat e_j)}
      + \| d_{\Tref\times \{0\}}^{3/2-s}\nabla u_1 \|_{L^2(\widehat e_j)}
    \right) + | \nabla u_1(\widehat\bv_4)| + \underbrace{| u_1(\what \bv_4)|}_{=0}\\
    &\lesssim | u |_{H^s(\Tref)} + \| d_{\widehat\cE}^{-s}u \|_{L^2(\Tref)}.
    \end{split}
  \end{align}
  As before,
  \begin{align}\label{eq:lemma:lifting-from-triangle:bc:19}
    | \nabla (u_2\circ \Pi_{\widehat f_j})(x,y,z)| \sim | (\nabla u_2)\circ \Pi_{\widehat f_j}(x,y,z)|.
  \end{align}
  We employ~\eqref{eq:lemma:lifting-from-triangle:bc:18},~\eqref{eq:lemma:lifting-from-triangle:bc:13}~\eqref{eq:lemma:lifting-from-triangle:bc:2} (with $\alpha=-1,\beta=1/2-s$ as well as $\alpha=0,\beta=1/2-s$) and
  conclude with~\eqref{eq:lemma:lifting-from-triangle:bc:6}-\eqref{eq:lemma:lifting-from-triangle:bc:8}
  \begin{align*}
    \| d_{\Tref\times \{0\}}^{1/2-s} \nabla\cA_{\widehat\cE}u \|_{L^2(\tetref)}
    &\lesssim \| d_{\Tref\times \{0\}}^{1/2-s} \nabla u_2 \|_{L^2(\tetref)}\\
    &\qquad+\sum_{\widehat f_k\in \widehat\cF}
    \left( 
      \| d_{\widehat e_{3+k}}^{-1} d_{\Tref\times \{0\}}^{1/2-s} u_2\circ \Pi_{\widehat f_k} \|_{L^2(\tetref)}
      + \| d_{\Tref\times \{0\}}^{1/2-s} \left( \nabla u_2 \right)\circ \Pi_{\widehat f_k} \|_{L^2(\tetref)}
    \right)\\
    &\lesssim \| d_{\Tref\times \{0\}}^{1/2-s} \nabla u_2 \|_{L^2(\tetref)}
    +\sum_{\widehat f_k\in \widehat\cF}
    \left( 
    \| d_{\Tref\times \{0\}}^{-s} u_2 \|_{L^2(\widehat f_k)}
    + \| d_{\Tref\times \{0\}}^{1-s} \nabla u_2 \|_{L^2(\widehat f_k)}
    \right)\\
    &\lesssim | u |_{H^s(\Tref)} + \| d_{\widehat\cE}^{-s}u \|_{L^2(\Tref)}.
  \end{align*}

  \textbf{Proof of~(\ref{item:lemma:lifting-from-triangle:bc-6}).}
  This follows by the consecutive construction of $\cA_{\what\cE}u$ via $u_1$ and $u_2$, using orthogonal projections onto lateral edges and faces
  and Lemma~\ref{lemma:lifting-from-triangle},~(\ref{item:lemma:lifting-from-triangle-7}).

  \textbf{Proof of~(\ref{item:lemma:lifting-from-triangle:bc-7}).} First, let $\delta>\widetilde\varepsilon>0$. Using
  Lemma~\ref{lemma:lifting-from-triangle}~(\ref{item:lemma:lifting-from-triangle-7}) and~(\ref{item:lemma:lifting-from-triangle-8}), we conclude
  \begin{align}\label{eq:lemma:lifting-from-triangle:bc:21}
    \| \partial^\bk u_1 (\cdot,z) \|_{L^{\infty}( (1-z)\Tref_{\widetilde\varepsilon})}
    \lesssim \| u \|_{L^2(\Tref)} + \| u \|_{W^{\abs{\bk},\infty}(\Tref_{\delta})}
  \end{align}
  with an implied constant depending only on $\delta-\widetilde\varepsilon$.
  We proceed to consider $u_2$. Using the triangle inequality and estimate~\eqref{eq:lemma:lifting-from-triangle:bc:21},
  it suffices to consider the correction term
  \begin{align*}
    t_j (x,y,z) := \frac{z}{p_j(x,y,z)} u_1(\Pi_{\widehat e_j}(x,y,z))
  \end{align*}
  for $j\in N$ and $(x,y,z)$ with $(x,y)\in (1-z)\Tref_{\widetilde\varepsilon}$.
  With~\eqref{eq:lemma:lifting-from-triangle:bc:10},~\eqref{eq:lemma:lifting-from-triangle:bc:11},
  and~\eqref{eq:lemma:lifting-from-triangle:bc:9} we conclude
  \begin{align}\label{eq:lemma:lifting-from-triangle:bc:23}
    |t_j (x,y,z)| \lesssim |u_1(\Pi_{\widehat e_j}(x,y,z))|\qquad \text{ and }\qquad
    |\nabla t_j (x,y,z)| \lesssim \frac{|u_1(\Pi_{\widehat e_j}(x,y,z))|}{d_{\widehat\bv_j}(x,y,z)} + |(\nabla u_1)(\Pi_{\widehat e_j}(x,y,z))|.
  \end{align}
  We distinguish two cases for $j\in N$:
  \begin{itemize}
    \item $j\in\left\{ 2,3 \right\}$. In this case, note that for $(x,y,z)\in\tetref$ with $(x,y)\in (1-z)\Tref_{\widetilde\varepsilon}$
      it holds $\Pi_{\widehat e_j}(x,y,z)\in \overline{\tetref_{\varepsilon_1}}$ for some $\varepsilon_1$ depending only on $\widetilde\varepsilon$,
      as well as $d_{\widehat\bv_j}(x,y,z)\gtrsim 1$. We conclude with~\eqref{eq:lemma:lifting-from-triangle:bc:23}
      \begin{align*}
	|\partial^{\bk} t_j (x,y,z)| \lesssim \| u_1 \|_{W^{\abs{\bk},\infty}(\tetref_{\varepsilon_1})},
      \end{align*}
      with a constant depending only on $\widetilde\varepsilon$.
      To bound further the right-hand side, we may use point~(\ref{item:lemma:lifting-from-triangle:bc-6}) of the present lemma in the
      particular case of $\what\cE=\emptyset$ as in this case $u_1=\cA_{\emptyset}u$ so that
      \begin{align*}
	\| u_1 \|_{W^{\abs{\bk},\infty}(\tetref_{\varepsilon_1})} = 
	\| \cA_{\emptyset} u \|_{W^{\abs{\bk},\infty}(\tetref_{\varepsilon_1})} \lesssim \| u \|_{L^2(\Tref)}.
    \end{align*}
    \item $j=1$.
      For $\abs{\bk}=0$, we use the first estimate of~\eqref{eq:lemma:lifting-from-triangle:bc:23} and~\eqref{eq:lemma:lifting-from-triangle:bc:21},
      as for $(x,y,z)\in\tetref$ with $(x,y)\in (1-z)\Tref_{\widetilde\varepsilon}$ it holds
      \begin{align*}
	|t_1(x,y,z)|\lesssim \| u_1 \|_{L^\infty(\widehat e_1)} \leq \sup_{z\in[0,1]} \| u_1(\cdot,z) \|_{L^\infty( (1-z)\Tref_{\widetilde\varepsilon})}
	\lesssim \| u \|_{L^2(\Tref)} + \| u \|_{L^\infty(\Tref_{\delta})}.
      \end{align*}
      If $\abs{\bk}=1$, we note that our assumptions yield $u(\widehat\bv_1)=0$, and Poincar\'e's inequality shows
      \begin{align*}
	|u_1(\Pi_{\widehat e_1}(x,y,z))| \lesssim d_{\widehat\bv_1}(\Pi_{\widehat e_1}(x,y,z)) \| \nabla u_1 \|_{L^{\infty}(\widehat e_1)}
      \end{align*}
      as well as $d_{\widehat\bv_1}(\Pi_{\widehat e_1}(x,y,z))\leq d_{\widehat \bv_1}(x,y,z)$. Finally,
      we conclude with the second estimate of~\eqref{eq:lemma:lifting-from-triangle:bc:23} for
      $(x,y,z)\in\tetref$ with $(x,y)\in (1-z)\Tref_{\widetilde\varepsilon}$
      and~\eqref{eq:lemma:lifting-from-triangle:bc:21}
      \begin{align*}
	|\partial^{\bk} t_j (x,y,z)| \lesssim \| \nabla u_1 \|_{L^\infty(\widehat e_1)}
	\leq \sup_{z\in[0,1]} \| \nabla u_1(\cdot,z) \|_{L^\infty( (1-z)\Tref_{\widetilde\varepsilon})}
	\lesssim \| u \|_{L^2(\Tref)} + \| u \|_{W^{1,\infty}(\Tref_{\delta})}.
    \end{align*}
  \end{itemize}
  We arrive at
  \begin{align}\label{eq:lemma:lifting-from-triangle:bc:20} 
    \| \partial^\bk u_2 (\cdot,z) \|_{L^{\infty}( (1-z)\Tref_{\widetilde\varepsilon})}
    \lesssim \| u \|_{L^2(\Tref)} + \| u \|_{W^{\abs{\bk},\infty}(\Tref_{\delta})}
  \end{align}
  for any $\widetilde\varepsilon$ and $\delta>\widetilde\varepsilon$, with a constant depending on $\widetilde\varepsilon$ as well as
  $\delta-\widetilde\varepsilon$.
  Finally, let $\varepsilon>0$.
  For $\widehat f_k\in\widehat\cF$ consider the correction term 
  \begin{align*}
    r_k(x,y,z) := \frac{z}{p_{\widehat f_k}(x,y)} u_2(\Pi_{\widehat f_k}(x,y,z))
  \end{align*}
  for $(x,y,z)$ with $(x,y)\in (1-z)\Tref_{\varepsilon}$.
  With~\eqref{eq:lemma:lifting-from-triangle:bc:18},~\eqref{eq:lemma:lifting-from-triangle:bc:13}, and~\eqref{eq:lemma:lifting-from-triangle:bc:19} we conclude that
  \begin{align*}
    |r_k(x,y,z)| \lesssim |u_2(\Pi_{\widehat f_k}(x,y,z))|\quad\text{ and }\qquad
    |\nabla r_k(x,y,z)| \lesssim \frac{|u_2(\Pi_{\widehat f_k}(x,y,z))|}{d_{\widehat e_{k+3}}(x,y,z)} + |(\nabla u_2)(\Pi_{\widehat f_k}(x,y,z))|.
  \end{align*}
  Again we distinguish two cases.
  \begin{itemize}
    \item $k=1$. In this case, note that for $(x,y,z)\in\tetref$ with $(x,y)\in (1-z)\Tref_{\varepsilon}$
      it holds $\Pi_{\widehat f_k}(x,y,z)\in \overline{\tetref_{\varepsilon_2}}$ for some $\varepsilon_2$ depending only on $\varepsilon$,
      as well as as well as $d_{\widehat e_{k+3}}(x,y,z)\gtrsim 1$. Hence, using again point~(\ref{item:lemma:lifting-from-triangle:bc-6})
      of the present lemma, we conclude
      \begin{align*}
	|\partial^{\bk} r_k(x,y,z)| \lesssim \| u_2 \|_{W^{\abs{\bk},\infty}(\tetref_{\varepsilon_2})} \lesssim \| u \|_{L^2(\Tref)},
      \end{align*}
      with a constant depending only on $\varepsilon$.
    \item $k\in\left\{ 2,3 \right\}$. In this case, note that for $\varepsilon$ sufficiently small there is some $\widetilde\varepsilon>\varepsilon$
      such that for $(x,y,z)\in \tetref$ with $(x,y)\in (1-z)\Tref_\varepsilon$ it holds
      for $ (\wilde x,\wilde y,\wilde z) = \Pi_{\widehat f_k}(x,y,z)$ that $(\wilde x,\wilde y)\in \overline{(1- \wilde z)\Tref_{\widetilde \varepsilon}}$.
      For $\abs{\bk}=1$, given that $u_2$ vanishes on $\widehat e_k$,
      Poincar\'e's inequality shows
      \begin{align*}
	|u_2(\Pi_{\widehat f_k}(x,y,z))| \lesssim d_{\widehat e_{k+3}}(\Pi_{\widehat f_k}(x,y,z))
	\max_{\wilde z\in[0,1]} \| \nabla u_2 \|_{L^\infty( \widehat f_k\cap\overline{(1-\wilde z)\Tref_{\widetilde\varepsilon}})}
      \end{align*}
      as well as $d_{\widehat e_{k+3}}(\Pi_{\widehat f_k}(x,y,z)) \leq d_{\widehat e_{k+3}}(x,y,z)$. Finally,
      for $\abs{\bk}\in\left\{ 0,1 \right\}$, we conclude with~\eqref{eq:lemma:lifting-from-triangle:bc:20}
      \begin{align*}
	|\partial^{\bk} r_k(x,y,z)| \lesssim 
	\max_{\wilde z\in[0,1]} \| \nabla^{\abs{\bk}} u_2 \|_{L^{\infty}( \widehat f_k\cap\overline{(1-\wilde z)\Tref_{\widetilde\varepsilon}})}
	\lesssim \| u \|_{L^2(\Tref)} + \| u \|_{W^{\abs{\bk},\infty}(\Tref_{\delta})}.
      \end{align*}
    \end{itemize}
  Together with triangle inequality, this shows the stipulated estimate.
\end{proof}
We define a reference prism $\Pref := \Tref\times(0,1)$, where $\Tref$ is the reference triangle.
\begin{lemma}\label{lemma:lifting-to-prism}
  Let $\what\cE \subset \{ \widehat e_{4}, \widehat e_5, \widehat e_6 \}$.
  There exists a linear operator $\cA^{\Pref}_{\widehat \cE}: L^1_{loc}(\Tref)\rightarrow C^{\infty}(\Pref)$ with the following properties:
  \begin{enumerate}[(i)]
    \item \label{item:lemma:lifting-to-prism:1} If $u$ is continuous at a point $\bx\in\Tref$, then $(\cA^{\Pref}_{\what\cE} u)(\bx,0)=u(\bx)$.
    \item \label{item:lemma:lifting-to-prism:2} The operator $\cA^{\Pref}_{\what\cE}:L^2(\Tref)\rightarrow L^2(\Pref)$ is bounded.
    \item \label{item:lemma:lifting-to-prism:3} If $u$ vanishes on an edge $\what e\in\widehat\cE$, then
      $\cA^{\Pref}_{\what\cE} u$ vanishes on the face $\what e \times (0,1)$.
    \item \label{item:lemma:lifting-to-prism:7} If $u$ is a polynomial of degree $p\geq|\what\cE|$ that vanishes
      on all edges in $\what\cE$, then $\cA^{\Pref}_{\what\cE}(\cdot,\cdot,z)$ is a polynomial of degree $p$ for fixed $z$.
    \item \label{item:lemma:lifting-to-prism:4} $\cA^{\Pref}_{\what\cE}u$ vanishes on the top face $\Tref\times {1}$.
    \item \label{item:lemma:lifting-to-prism:5} 
      For every $\theta\in(0,1)$ there is a constant $C_\theta>0$ such that
      \begin{align*}
	\| d_{\Tref\times \{0\}}^{1/2-\theta} \nabla \cA^{\Pref}_{\widehat\cE}u \|_{L^2(\Pref)}
	\leq C_\theta \left( |u|_{H^\theta(\Tref)} + \| d_{\widehat \cE}^{-\theta}u \|_{L^2(\Tref)}  \right).
      \end{align*}
    \item \label{item:lemma:lifting-to-prism:6}
      For $\varepsilon>0$ sufficiently small there is a $\delta>0$ with $\delta>\varepsilon$ such that
      for {$\bk\in\N_0^3$ with $\abs{\bk}\leq 1$ there holds for any $u$ vanishing on $\widehat\cE$ that}
      \begin{align*}
	\| \partial^\bk \cA^{\Pref}_{\widehat\cE} u (\cdot,z) \|_{L^{\infty}( \Tref_\varepsilon)}
	\lesssim \| u \|_{L^2(\Tref)} + \| u \|_{W^{\abs{\bk},\infty}(\Tref_\delta)} \qquad\text{ for all }z \in [0,1],
      \end{align*}
      with implied constant depending only on $\varepsilon$.
  \end{enumerate}
\end{lemma}
\begin{proof}
  Denote by
  \begin{align*}
    T_{\cD}^{3D}: \Pref \rightarrow \tetref,\quad
    \begin{pmatrix}
      x\\y\\z
    \end{pmatrix} \rightarrow
    \begin{pmatrix}
      x(1-z)\\y(1-z)\\z
    \end{pmatrix}
  \end{align*}
  the Duffy transform and note $|\det d T_{\cD}^{3D}| = (1-z)^2$.
  Let $\cA_{\what\cE}$ be the operator of Lemma~\ref{lemma:lifting-from-triangle:bc} and define
  \begin{align*}
    (\cA^{\Pref}_{\widehat \cE} u)(x,y,z) := (1-z)(A_{\what\cE} u)\circ T_{\cD}^{3D}(x,y,z).
  \end{align*}
  \textbf{Proof of~(\ref{item:lemma:lifting-to-prism:1}).} This follows immediately from
  Lemma~\ref{lemma:lifting-from-triangle:bc} (\ref{item:lemma:lifting-from-triangle:bc-1}).

  \textbf{Proof of~(\ref{item:lemma:lifting-to-prism:2}).}
  This follows from Lemma~\ref{lemma:lifting-from-triangle:bc} (\ref{item:lemma:lifting-from-triangle:bc-2}) by substitution.

  \textbf{Proof of~(\ref{item:lemma:lifting-to-prism:3}).}
  This follows from Lemma~\ref{lemma:lifting-from-triangle:bc} (\ref{item:lemma:lifting-from-triangle:bc-3})
  and the fact that the Duffy transform maps the corresponding face of $\tetref$ where $\cA_{\what \cE}u$ vanishes to $\what e\times (0,1)$.

  \textbf{Proof of~(\ref{item:lemma:lifting-to-prism:7}).}
  This follows from Lemma~\ref{lemma:lifting-from-triangle:bc} (\ref{item:lemma:lifting-from-triangle:bc-4}) and the definition of $T_{\cD}^{3D}$.

  \textbf{Proof of~(\ref{item:lemma:lifting-to-prism:4}).}
  This follows by construction.

  \textbf{Proof of~(\ref{item:lemma:lifting-to-prism:5}).}
  By the product and chain rule,
  \begin{align}\label{eq:lemma:lifting-to-prism:1}
    \nabla \cA^{\Pref}_{\widehat\cE}u = 
    \begin{pmatrix}
      0\\0\\-\cA_{\what\cE}u\circ T_{\cD}^{3D}
    \end{pmatrix}
    + (1-z)\bigl(dT_{\cD}^{3D}\bigr)^\top \bigl( \nabla \cA_{\what\cE}u \bigr)\circ T_{\cD}^{3D}.
  \end{align}
  To bound the first term, we choose some $\varepsilon>0$ and calculate
  \begin{align*}
    \int_{\Pref}& z^{1-2\theta}|\cA_{\what\cE}u\circ T_{\cD}^{3D}(x,y,z)|^2\,dxdydz
    = \int_{\tetref} z^{1-2\theta}|\cA_{\what\cE}u(x,y,z)|^2\frac{1}{(1-z)^2}\,dxdydz\\
    &\lesssim \int_{\tetref\setminus\tetref_\varepsilon} z^{1-2\theta}|\cA_{\what\cE}u(x,y,z)|^2\,dxdydz +
    \| \cA_{\what\cE}u \|_{L^\infty(T^{3D}_\varepsilon)}^2
    \int_{\tetref_\varepsilon} \frac{1}{(1-z)^2}\,dxdydz\\
    &\lesssim \| u \|_{L^2(\Tref)}^2,
  \end{align*}
  where the last estimate follows from Lemma~\ref{lemma:lifting-from-triangle:bc} (\ref{item:lemma:lifting-from-triangle:bc-2})
  and (\ref{item:lemma:lifting-from-triangle:bc-6}). To bound the second term in~\eqref{eq:lemma:lifting-to-prism:1}, we use
  $\| dT_{\cD}^{3D} \|_2 \lesssim 1$ and substitution
  \begin{align*}
    \int_{\Pref} z^{1-2\theta} (1-z)^2|\bigl(\nabla \cA_{\what\cE}u\bigr)\circ T_{\cD}^{3D}(x,y,z)|^2\,dxdydz
    &\lesssim \int_{\tetref} z^{1-2\theta} |\nabla \cA_{\what\cE}u (x,y,z)|^2\,dxdydz\\
    &\lesssim |u|_{H^\theta(\Tref)}^2 + \| d_{\widehat \cE}^{-\theta}u \|_{L^2(\Tref)}^2,
  \end{align*}
  where the last estimate follows from Lemma~\ref{lemma:lifting-from-triangle:bc} (\ref{item:lemma:lifting-from-triangle:bc-5}).

  \textbf{Proof of~(\ref{item:lemma:lifting-to-prism:6}).}
  For $\abs{\bk}=0$, this follows immediately from Lemma~\ref{lemma:lifting-from-triangle:bc}~(\ref{item:lemma:lifting-from-triangle:bc-7}),
  taking into account that
  \begin{align*}
    \| \cA^{\Pref}_{\widehat\cE} u (\cdot,z) \|_{L^{\infty}( \Tref_\varepsilon)} =
    (1-z)\| \cA_{\widehat\cE} u (\cdot,z) \|_{L^{\infty}( (1-z)\Tref_\varepsilon)}.
  \end{align*}
  For $\abs{\bk}=1$, we additionally use the formula~\eqref{eq:lemma:lifting-to-prism:1}.
\end{proof}
\subsection{Liftings for decomposed FEM spaces}\label{sec:lifting:dcomp}
\begin{lemma}\label{lemma:lift:elpart}
  Let $K\in\cT$ be an element and $u_K \in\wilde\cS^{\bp,1}(\cT|_K)$. Then there exists a function
  $v:[0,\infty)\rightarrow \wilde\cS^{\bp,1}(\cT|_K)$ such that $v(0) = u_K$ and, for $\theta \in (0,1)$,
  \begin{align*}
    \int_{0}^\infty t^{2(1-\theta)} \left( \| \nabla v(t) \|_{L^2(K)}^2 + \| v'(t) \|_{L^2(K)}^2 \right) \frac{dt}{t}
    \lesssim | u_K |_{H^\theta(K)}^2 + \| d_{\partial K}^{-\theta} u_K \|_{L^2(K)}^2.
  \end{align*}
\end{lemma}
\begin{proof}
  Set $\what u=u_K\circ F_K$ and $p = p_K$.
  First, suppose that $K$ is a square, i.e., $\what K = \Sref$, such that $\what u \in \wilde\cQ^p$.
  We apply Lemma~\ref{lemma:k-trace-space} (\ref{item:lemma:k-trace-space-i}) with
  $X_0 = (\wilde\cQ^p,\|\cdot\|_{L^2(\Sref)})$ and $X_1 = (\wilde\cQ^p,\|\cdot\|_{\wilde H^1(\Sref)})$.
  Accordingly, there exists a function $\what v: [0,\infty]\rightarrow \wilde\cQ^p$ such that
  \begin{align}\label{lemma:lift:elpart:1}
    \begin{split}
    \int_{0}^\infty t^{2(1-\theta)} \left( \| \nabla \what v(t) \|_{L^2(\Sref)}^2 + \| \what v'(t) \|_{L^2(\Sref)}^2 \right) \frac{dt}{t}
    &\lesssim | \what u |_{[X_0,X_1]_\theta}^2 \leq \| \what u \|_{[X_0,X_1]_\theta}^2\\
    &\lesssim \| \what u \|_{[L^2(\Sref),\wilde H^1(\Sref)]_\theta}^2\\
    &\lesssim | \what u |_{H^\theta(\Sref)}^2 + \| d_{\partial \Sref}^{-\theta} \what u \|_{L^2(\Sref)}^2.
    \end{split}
  \end{align}
  where the penultimate estimate follows from Proposition~\ref{prop:GL} (\ref{item:prop:GL-i}),
  and the last estimate from Lemma~\ref{lem:equiv:lm:inter}, $\rm(ii)$.
  If $K$ is a triangle, i.e., $\Kref = \Tref$, then in particular $p\geq 3$,
  and we use Lemma~\ref{lemma:lifting-to-prism} to define $\what v := \cA^{\Pref}_{\partial\Tref} \what u$
  on $\Pref$. Note that $\what v$ vanishes at $t=1$, and we extend it by zero on $\Tref\times[1,\infty)$.
  Due to Lemma~\ref{lemma:lifting-to-prism}~\eqref{item:lemma:lifting-to-prism:3}
  and~\eqref{item:lemma:lifting-to-prism:7} it holds $\what v:[0,\infty]\rightarrow \wilde\cP^p$.
  Then, due to Lemma~\ref{lemma:lifting-to-prism}~\eqref{item:lemma:lifting-to-prism:5},
  \begin{align}\label{lemma:lift:elpart:2}
    \begin{split}
    \int_{0}^\infty t^{2(1-\theta)} \left( \| \nabla \what v(t) \|_{L^2(\Tref)}^2 + \| \what v'(t) \|_{L^2(\Tref)}^2 \right) \frac{dt}{t}
    &= \| d_{\Tref\times\{0\}}^{1/2-\theta} \nabla\cA^{\Pref}_{\partial\Tref}\what u \|_{L^2(\Pref)}^2\\
    &\lesssim |\what u|_{H^\theta(\Tref)}^2 + \| d_{\partial\Tref}^{-\theta}\what u \|_{L^2(\Tref)}^2.
    \end{split}
  \end{align}
  Independently of the shape of $K$ we define $v(t) := \what v(t/h_K)\circ F_K^{-1}$,
  such that $v:[0,\infty)\rightarrow \wilde\cS^{\bp,1}(\cT|_K)$.
  Scaling arguments transform
  the left hand sides of~\eqref{lemma:lift:elpart:1} and~\eqref{lemma:lift:elpart:2} into
  \begin{align*}
    h_K^{2\theta-2}\int_{0}^\infty t^{2(1-\theta)} \left( \| \nabla v(t) \|_{L^2(K)}^2 + \| v'(t) \|_{L^2(K)}^2 \right) \frac{dt}{t},
  \end{align*}
  while the right-hand sides transform into
  \begin{align*}
    h^{2\theta-2}\left( | u_K |_{H^\theta(K)}^2 + \| d_{\partial K}^{-\theta} u_K \|_{L^2(K)}^2\right).
  \end{align*}
  The stated estimate follows.
\end{proof}
\begin{lemma}\label{lemma:lift:patches}
  Let $\omega$ be a vertex or edge patch in $\cT$. If $\omega$ is a vertex patch, set $\widehat\cE=\left\{ \what e_6 \right\}$.
  If $\omega$ is an edge patch, set $\what\cE=\left\{ \what e_5,\what e_6 \right\}$. With the notation introduced in Lemma~\ref{lem:kmr:decomp:2},
  let $\wilde u\in X_V$ or $\wilde u\in X_e$.
  Then there exists a function
  $v:[0,\infty)\rightarrow \wilde\cS^{\bp,1}(\cT|_\omega)$ such that $v(0) = T_\omega \wilde u$ and, for $\theta \in (0,1)$,
  \begin{align*}
    h_\omega^{-2+2\theta}\int_{0}^\infty t^{2(1-\theta)} \left( \| \nabla v(t) \|_{L^2(\omega)}^2 + \| v'(t) \|_{L^2(\omega)}^2 \right) \frac{dt}{t}
    \lesssim |\wilde u|_{H^\theta(\Tref)}^2 + \| d_{\widehat \cE}^{-\theta} \wilde u \|_{L^2(\Tref)}^2
    + \| \wilde u \|_{W^{1,\infty}(\Tref_\delta)}^2,
  \end{align*}
  where $\delta>0$ is chosen according to Lemma~\ref{lemma:lifting-to-prism} (vii).
\end{lemma}
\begin{proof}
  Let $\varepsilon>0$ be chosen according to Lemma~\ref{lemma:lifting-to-prism}.
  Set $\wilde v(t) := \cA^{\Pref}_{\what\cE} \wilde u(t)$. Note
  that $\wilde v$ vanishes at $t=1$, and we extend it by zero on $\Tref\times[1,\infty)$. Then set
  $v(t) := T_{\omega} \wilde v(t/h_\omega)$.
  We have $v:[0,\infty)\rightarrow \wilde\cS^{\bp,1}(\cT|_\omega)$ and $v(0)=T_\omega((\cA^P_{\what\cE}\wilde u)(0)) = T_\omega \wilde u$.
  Due to Lemma~\ref{lemma:lifting-to-prism}, we conclude
  \begin{align*}
   h_{\omega}^{2-2\theta} \int_{0}^\infty &t^{2(1-\theta)} \left( \| \nabla v(t) \|_{L^2(\omega)}^2 + \| v'(t) \|_{L^2(\omega)}^2 \right) \frac{dt}{t}\\
    &= \int_0^1 t^{2(1-\theta)}
    \left( \| \nabla T_\omega \wilde v(t) \|_{L^2(\omega)}^2  + h_\omega^{-2}\| T_\omega \wilde v'(t) \|_{L^2(\omega)}^2 \right) \frac{dt}{t}\\
    &\lesssim  \int_0^1 t^{2(1-\theta)}
    \left( 
    \| \nabla \wilde v(t) \|_{L^2(\Tref)}^2 + \| \nabla \wilde v(t) \|_{L^\infty(\Tref_\varepsilon)}^2 +
    \| \wilde v'(t) \|_{L^2(\Tref)}^2 + \| \wilde v'(t) \|_{L^\infty(\Tref_\varepsilon)}^2
    \right) \frac{dt}t\\
    &= \| d_{\Tref\times \{0\}}^{1/2-\theta} \nabla \cA^{\Pref}_{\widehat\cE}\wilde u \|_{L^2(\Pref)}^2
    + \int_0^1 t^{2(1-\theta)} \| \nabla \wilde v(t) \|_{L^\infty(\Tref_\varepsilon)}^2 +
    \| \wilde v'(t) \|_{L^\infty(\Tref_\varepsilon)}^2 \frac{dt}t\\
    &\lesssim  |\wilde u|_{H^\theta(\Tref)}^2 + \| d_{\widehat \cE}^{-\theta} \wilde u \|_{L^2(\Tref)}^2
    + \int_0^1 t^{2(1-\theta)} \| \nabla \wilde v(t) \|_{L^\infty(\Tref_\varepsilon)}^2 +
    \| \wilde v'(t) \|_{L^\infty(\Tref_\varepsilon)}^2 \frac{dt}t\\
    &\lesssim  |\wilde u|_{H^\theta(\Tref)}^2 + \| d_{\widehat \cE}^{-\theta} \wilde u \|_{L^2(\Tref)}^2
    + \| \wilde u \|_{L^2(\Tref)}^2 + \| \nabla \wilde u \|_{L^\infty(\Tref_\delta)}^2 + \| \wilde u \|_{L^\infty(\Tref_\delta)}^2.
  \end{align*}
  Here, the first estimate follows from Lemma~\ref{lem:kmr:decomp:2},~(ii), the second one from
  Lemma~\ref{lemma:lifting-to-prism},~(\ref{item:lemma:lifting-to-prism:5}), and the last one from
  Lemma~\ref{lemma:lifting-to-prism},~(\ref{item:lemma:lifting-to-prism:6}).
\end{proof}
\begin{proof}[Proof of Theorem~\ref{thm:main}]
  We will only treat the case of homogeneous boundary conditions, the general case follows along the
  same lines.
  The canonical continuous embeddings
  \begin{align*}
    (\wilde\cS^{\bp,1}(\cT),\| \cdot \|_{L^2(\Omega)}) &\subset L^2(\Omega),\\
    (\wilde\cS^{\bp,1}(\cT),\| \cdot \|_{\wilde H^1(\Omega)}) &\subset \wilde H^1(\Omega)
  \end{align*}
  immediately yield for all $u\in\wilde\cS^{\bp,1}(\cT)$ the estimate
  \begin{align*}
    \| u \|_{[ (\wilde \cS^{\bp,1}(\cT),\| \cdot \|_{L^2(\Omega)}), (\wilde \cS^{\bp,1}(\cT),\| \cdot \|_{\wilde H^1(\Omega)}) ]_{\theta}}
    \gtrsim \| u \|_{[L^2(\Omega),\wilde H^1(\Omega)]_\theta}.
  \end{align*}
  It therefore remains to show the converse estimate. To that end, we employ Lemma~\ref{lem:kmr:decomp:2} and write
  \begin{align*}
    u = u_1 + \sum_{V\in\cV^{\rm int}} T_{\omega_V}(\wilde u_V) + \sum_{e\in\cE^{\rm int}}T_{\omega_e}(\wilde u_e) + \sum_{K\in\cT} u_K.
  \end{align*}
  The map $u \mapsto u_1$ is linear and bounded in $L^2(\Omega)$ and $\wilde H^1(\Omega)$, and hence also in
  $[L^2(\Omega),\wilde H^1(\Omega)]_\theta$.
  As it maps into $\wilde\cS^{1,1}(\cT)$, we moreover obtain
  \begin{align*}
    \| u_1 \|_{[ (\wilde \cS^{1,1}(\cT),\| \cdot \|_{L^2(\Omega)}), (\wilde \cS^{1,1}(\cT),\| \cdot \|_{\wilde H^1(\Omega)}) ]_{\theta}}
    \lesssim \| u \|_{[L^2(\Omega),\wilde H^1(\Omega)]_\theta},
  \end{align*}
  and this also yields
  \begin{align*}
    \| u_1 \|_{[ (\wilde \cS^{\bp,1}(\cT),\| \cdot \|_{L^2(\Omega)}), (\wilde \cS^{\bp,1}(\cT),\| \cdot \|_{\wilde H^1(\Omega)}) ]_{\theta}}
    \lesssim \| u \|_{[L^2(\Omega),\wilde H^1(\Omega)]_\theta}.
  \end{align*}
  We conclude that it remains to show that
    \begin{align*}
      \| u-u_1 \|_{[ (\wilde \cS^{\bp,1}(\cT),\| \cdot \|_{L^2(\Omega)}), (\wilde \cS^{\bp,1}(\cT),\| \cdot \|_{\wilde H^1(\Omega)}) ]_{\theta}}
      \lesssim \| u \|_{[L^2(\Omega),\wilde H^1(\Omega)]_\theta}.
    \end{align*}
  For $K\in\cT$, let $v_K$ be the functions constructed from $u_K$ by
  Lemma~\ref{lemma:lift:elpart}. For $V\in\cV^{\rm int}$ and $e\in\cE^{\rm int}$ denote by $v_V$ and $v_e$ the functions
  constructed from $u_V$ and $u_e$ by Lemma~\ref{lemma:lift:patches}. Define
  \begin{align*}
    v = \sum_{V\in\cV^{\rm int}} v_V + \sum_{e\in\cE^{\rm int}}v_e + \sum_{K\in\cT} v_K
  \end{align*}
  and note that $v:[0,\infty)\rightarrow\wilde\cS^{\bp,1}(\cT)$ as well as $v(0)=u-u_1$. Furthermore,
  \begin{align*}
    \int_{0}^\infty t^{2(1-\theta)} &\left( \| \nabla v(t) \|_{L^2(\Omega)}^2 + \| v'(t) \|_{L^2(\Omega)}^2 \right) \frac{dt}{t}\\
    &\lesssim \sum_{V\in\cV^{\rm int}} \int_{0}^\infty t^{2(1-\theta)} \left( \| \nabla v_V(t) \|_{L^2(\omega_V)}^2 + \| v_V'(t) \|_{L^2(\omega_V)}^2 \right) \frac{dt}{t}\\
    &\qquad+ \sum_{e\in\cE^{\rm int}} \int_{0}^\infty t^{2(1-\theta)} \left( \| \nabla v_e(t) \|_{L^2(\omega_e)}^2 + \| v_e'(t) \|_{L^2(\omega_e)}^2 \right) \frac{dt}{t}\\
    &\qquad+ \sum_{K\in\cT} \int_{0}^\infty t^{2(1-\theta)} \left( \| \nabla v_K(t) \|_{L^2(K)}^2 + \| v_K'(t) \|_{L^2(K)}^2 \right) \frac{dt}{t}\\
    &\lesssim \sum_{V\in\cV^{\rm int}} h_{\omega_V}^{2-2\theta} \left( |\wilde u_V|_{H^\theta(\Tref)}^2
    + \| d_{\widehat e_6}^{-\theta} \wilde u_V \|_{L^2(\Tref)}^2
    + \| \wilde u_V \|_{W^{1,\infty}(\Tref_\delta)}^2 \right)\\
    &\qquad +  \sum_{e\in\cE^{\rm int}} h_{\omega_e}^{2-2\theta} \left( |\wilde u_e|_{H^\theta(\Tref)}^2
    + \| d_{\widehat e_5\cup \widehat e_6}^{-\theta} \wilde u_e \|_{L^2(\Tref)}^2
    + \| \wilde u_e \|_{W^{1,\infty}(\Tref_\delta)}^2 \right)\\
    &\qquad+ \sum_{K\in\cT} | u_K |_{H^\theta(K)}^2 + \| d_{\partial K}^{-\theta} u_K \|_{L^2(K)}^2\\
    &\lesssim \| u \|_{[L^2(\Omega),\wilde H^1(\Omega)]_\theta}^2.
  \end{align*}
  Here, the first estimate follows using the finite overlap of the supports of the involved patches and a coloring argument, the
  second one from Lemmas~\ref{lemma:lift:elpart} and~\ref{lemma:lift:patches},
  and the last one from Lemma~\ref{lem:kmr:decomp:2}.
  Hence, with Lemma~\ref{lemma:k-trace-space}~\eqref{item:lemma:k-trace-space-ii} we conclude that
  \begin{align*}
    | u-u_1 |_{[ (\wilde \cS^{\bp,1}(\cT),\| \cdot \|_{L^2(\Omega)}), (\wilde \cS^{\bp,1}(\cT),\| \cdot \|_{\wilde H^1(\Omega)}) ]_{\theta}}^2
    \leq \| u \|_{[L^2(\Omega),\wilde H^1(\Omega)]_\theta}^2,
  \end{align*}
  and Proposition~\ref{lemma:K-vs-k} shows
  \begin{align*}
    \| u - u_1 \|_{[ (\wilde \cS^{\bp,1}(\cT),\| \cdot \|_{L^2(\Omega)}), (\wilde \cS^{\bp,1}(\cT),\| \cdot \|_{\wilde H^1(\Omega)}) ]_{\theta}}
    &\lesssim \| u - u_1 \|_{L^2(\Omega)} +
    | u-u_1 |_{[ (\wilde \cS^{\bp,1}(\cT),\| \cdot \|_{L^2(\Omega)}), (\wilde \cS^{\bp,1}(\cT),\| \cdot \|_{\wilde H^1(\Omega)}) ]_{\theta}}\\
    &\lesssim \| u \|_{[L^2(\Omega),\wilde H^1(\Omega)]_\theta},
  \end{align*}
  which concludes the proof.
\end{proof}
\bibliographystyle{abbrv}
\bibliography{literature}
\end{document}